\documentclass[a4paper,10pt]{amsart}

\usepackage{amsmath, amsfonts, amsthm, amssymb, amscd, mathrsfs, tikz-cd, hyperref}

\usepackage[utf8]{inputenc}
\usepackage[a4paper,margin=1in]{geometry}
\usepackage{latexsym,amsmath,amsfonts,amscd,amssymb, hyperref}
\usepackage{graphicx}

\makeindex

\title[Generalized Quivers, \&c]{Generalized Quivers, Orthogonal and Symplectic Representations, and Hitchin-Kobayashi Correspondences}
\author{Artur de Araujo}
\address{Departamento de Matem\'atica da Faculdade de Ci\^encias, Universidade do Porto\\ Rua do Campo Alegre 687\\4169-007 Porto}
\email{aaraujo@fc.up.pt}
\date{}

\numberwithin{equation}{section}

\newcommand{\GL}{\mathrm{GL}}
\newcommand{\SL}{\mathrm{SL}}
\newcommand{\Or}{\mathrm{O}}
\newcommand{\SO}{\mathrm{SO}}
\newcommand{\Sp}{\mathrm{Sp}}
\newcommand{\U}{\mathrm{U}}

\newcommand{\Hom}{\mathrm{Hom}}
\newcommand{\End}{\mathrm{End}}

\newcommand{\tr}{\mathrm{tr}}
\newcommand{\rk}{\mathrm{rk}}

\newcommand{\ta}{t(\alpha)}
\newcommand{\ha}{h(\alpha)}
\newcommand{\sa}{\sigma (\alpha)}

\theoremstyle{definition}           		\newtheorem{plainquiver}{Definition}[section]
\theoremstyle{definition}           		\newtheorem{twistedbundles}[plainquiver]{Definition}

\theoremstyle{plain}           		\newtheorem{principalrep}[plainquiver]{Proposition}
\theoremstyle{definition}           		\newtheorem{classification1}{Classification Problem}
\theoremstyle{definition}           		\newtheorem{classification2}{Classification Problem}
\theoremstyle{palin}           		\newtheorem{flagcharacter}[plainquiver]{Theorem}
\theoremstyle{palin}           		\newtheorem{bundlefiltration}[plainquiver]{Theorem}
\theoremstyle{definition}           		\newtheorem{cst}[plainquiver]{Definition}
\theoremstyle{plain}           		\newtheorem{morphismsfactor}[plainquiver]{Proposition}
\theoremstyle{definition}           		\newtheorem{quiversheaf}[plainquiver]{Definition}
\theoremstyle{plain}           		\newtheorem{simplifiedplainstability}[plainquiver]{Proposition}
\theoremstyle{definition}           		\newtheorem{simple}[plainquiver]{Definition}
\theoremstyle{plain}           		\newtheorem{hksimplequiverreps}[plainquiver]{Theorem}
\theoremstyle{definition}           		\newtheorem{psreps}[plainquiver]{Definition}
\theoremstyle{plain}           		\newtheorem{hkpsquivers}[plainquiver]{Theorem}

\theoremstyle{definition}           		\newtheorem{generalizedquiver}[plainquiver]{Definition}
\theoremstyle{definition}           		\newtheorem{twistedgenquiver}[plainquiver]{Definition}
\theoremstyle{definition}           		\newtheorem{generalizedbundle}[plainquiver]{Definition}

\theoremstyle{definition}           		\newtheorem{symquiver}[plainquiver]{Definition}
\theoremstyle{plain}           		\newtheorem{derksenweyman}[plainquiver]{Theorem}
\theoremstyle{plain}           		\newtheorem{dwequivariance}[plainquiver]{Corollary}
\theoremstyle{definition}           		\newtheorem{orthbundle}[plainquiver]{Definition}
\theoremstyle{definition}           		\newtheorem{remarktwisting}[plainquiver]{Remark}
\theoremstyle{plain}           		\newtheorem{orthsubspace}[plainquiver]{Lemma}
\theoremstyle{plain}           		\newtheorem{genortheq}[plainquiver]{Lemma}
\theoremstyle{plain}           		\newtheorem{symplectomorphism}[plainquiver]{Theorem}
\theoremstyle{definition}			\newtheorem{genorthst}[plainquiver]{Definition}
\theoremstyle{plain}           		\newtheorem{morphismsfactororth}[plainquiver]{Theorem}
\theoremstyle{definition}			\newtheorem{slopestabilityorth}[plainquiver]{Definition}
\theoremstyle{plain}           		\newtheorem{equivalencestability}[plainquiver]{Proposition}
\theoremstyle{definition}			\newtheorem{orthss}[plainquiver]{Definition}
\theoremstyle{plain}           		\newtheorem{orthplainrelation}[plainquiver]{Theorem}
\theoremstyle{plain}           		\newtheorem{simpleorthstable}[plainquiver]{Corollary}
\theoremstyle{plain}           		\newtheorem{onlyfixed}[plainquiver]{Corollary}
\theoremstyle{plain}           		\newtheorem{psdecomp}[plainquiver]{Lemma}
\theoremstyle{definition}			\newtheorem{riemanncase}[plainquiver]{Remark}
\theoremstyle{plain}           		\newtheorem{completesolution}[plainquiver]{Theorem}

\theoremstyle{definition}			\newtheorem{dualizing}[plainquiver]{Definition}
\theoremstyle{plain}			\newtheorem{supermixedgeneralized}[plainquiver]{Theorem}
\theoremstyle{definition}			\newtheorem{generalizedsupermixed}[plainquiver]{Definition}
\theoremstyle{definition}			\newtheorem{supermixed}[plainquiver]{Definition}
\theoremstyle{plain}			\newtheorem{geometricsupermixed}[plainquiver]{Theorem}
\theoremstyle{definition}			\newtheorem{pathalgebra}[plainquiver]{Remark}
\theoremstyle{plain}			\newtheorem{correspondencesupermixed}[plainquiver]{Theorem}
\theoremstyle{plain}			\newtheorem{supermixedbundle}[plainquiver]{Theorem}

\theoremstyle{definition}           		\newtheorem{mixed}[plainquiver]{Definition}
\theoremstyle{definition}           		\newtheorem{qmixedrep}[plainquiver]{Definition}

\theoremstyle{plain}           		\newtheorem{hkcorrespondence}[plainquiver]{Theorem}
\theoremstyle{definition}           		\newtheorem{generalsimplicity}[plainquiver]{Definition}
\theoremstyle{definition}           		\newtheorem{generalstability}[plainquiver]{Definition}
\theoremstyle{plain}           		\newtheorem{integralproperties}[plainquiver]{Proposition}
\theoremstyle{plain}           		\newtheorem{criticalintegral}[plainquiver]{Lemma}
\theoremstyle{plain}           		\newtheorem{equivalencenorms}[plainquiver]{Lemma}
\theoremstyle{definition}           		\newtheorem{mainest}[plainquiver]{Definition}
\theoremstyle{plain}           		\newtheorem{equivalenceestimates}[plainquiver]{Corollary}
\theoremstyle{plain}           		\newtheorem{properintegral}[plainquiver]{Lemma}
\theoremstyle{plain}           		\newtheorem{contradictingsequence}[plainquiver]{Lemma}

\begin{document}

\begin{abstract}
We review the theory of quiver bundles over a K\"ahler manifold, and then introduce the concept of generalized quiver bundles for an arbitrary reductive group $G$. We first study the case when $G=\Or (V)$ or $\Sp (V)$, interpreting them as orthogonal (resp. symplectic) bundle representations of the symmetric quivers introduced by Derksen-Weyman. We also study supermixed quivers, which simultaneously involve both orthogonal and symplectic symmetries. Finally, we discuss Hitchin-Kobayashi correspondences for these objects.
\end{abstract}

\maketitle

\section*{Introduction}

The main aim of this paper is to pave the way for a Lie theoretic aproach to the theory of quiver bundles with arbitrary symmetries. Quivers are geometric objects introduced by Gabriel in the 1970's, and have since found applications in a wide range of areas, including geometric representation theory, gauge theory, and mirror symmetry. They have also been used in theoretical physics (in the bundle case, under the name of gauged quivers.) Let $\mathscr{A}$ be the category of finite dimensional vector spaces, or the category of (holomorphic) vector bundles over some K\"ahler manifold $X$ (more generally we could take any category, but these are the ones of most eminent interest.) The definitions are as follows.

\begin{plainquiver} 
\begin{enumerate}
\item A \emph{quiver} $Q$ is a finite diagram in the sense of Barr; alternatively, it is a finite directed graph, with set of vertices $I$, and set of arrows $A$. We let $t:A\to I$ and $h:A\to I$ be the tail and head functions, respectively.

\item A \emph{representation} $V$ of $Q$ is a realization of the diagram $Q$ in $\mathscr{A}$; equivalently, a representation is an assignment of an object $V_i$ for each vertex $i\in I$, and a morphism $\phi_\alpha: V_{t(\alpha)}\to V_{h(\alpha)}$ for every arrow $\alpha \in A$.
\end{enumerate}
\end{plainquiver}

When $\mathscr{A}$ is the category of vector spaces, we will specify it by speaking of ``finite dimensional representations.'' To a representation in either category, we will refer as ``plain representation,'' to distinguish them from the representations with additional symmetries that we'll consider later on.

Ever since Hitchin introduced Higgs bundles, it has been aparent the need to expand the previous concept of representation to include twistings.

\begin{twistedbundles}
\begin{enumerate}
\item A \emph{twisting} for a quiver $Q$ is a choice of a vector bundle $\mathbb{M}_\alpha$ over $X$ for each arrow in $Q$.

\item A \emph{twisted quiver bundle representation} of $Q$ (or $Q$-bundles, for short,) is a choice of a vector bundle $\mathbb{V}_i$ for each vertex $i$, and a morphism $\phi_\alpha: \mathbb{V}_{\ta}\otimes \mathbb{M}_\alpha \to \mathbb{V}_{\ha}$ for each arrow.
\end{enumerate}
\end{twistedbundles}

Each representation determines two vectors, indexed by the set of vertices and morphisms, the entries of which are respectively the rank of the vector bundle over the given vertex, and the rank of the twisting bundle associated with the given arrow. We denote by $\underline{i}$ and $\underline{\alpha}$ the prescribed dimension vectors for the vector bundles over the vertices and morphims, respectively.

On each vector bundle acts its gauge group, the group of automorphisms of the vector bundle covering the identity map on the base space. If we take the product of the gauge groups of vertices, we have an obvious action on the space of $Q$-bundles just defined. In the finite dimensional case, this is the action of a closed linear group on an affine variety, and therefore falls within the scope of GIT. A study along these lines was undertaken in \cite{king}; a useful reference is also \cite{reineke}. Quiver bundles (often going over to quiver sheaves) have also been amply studied, including an inifinitesimal study of the moduli space in \cite{gk}, and a complete Hitchin-Kobayashi correspondence in \cite{ag}. We quickly review the Hitchin Kobayashi correspondence in this case in section \ref{plain}.

However, very little is known about the representations of quivers with additional symmetries. This is the direction that the present paper takes. Derksen and Weyman introduced in \cite{derksen} a Lie theoretic generalization of quivers that is suitable for the discussion of an arbitrary (reductive) group of symmetries. They called it generalized quivers, and in section \ref{generalized}, we take an apropriate point of view to extend it to the bundle sitation.

Derksen-Weyman also completely characterized the representations of generalized quivers when the reductive group in question is either the orthogonal or the symplectic group. This result can be generalized to the bundle case as well, a Hitchin-Kobayashi correspondence can easily be computed for this case, and it is also possible to relate it to the plain case. This is done in section \ref{orthreps}.

Supermixed quivers are quivers with symmetries intermixing both the orthogonal and symplectic groups, and therefore generalizing symmetric quivers. They were introduced by Zubkov \cite{zubkov1} \cite{zubkov2} \cite{mixed}, who also computed their invariants. Section \ref{supermixed} is devoted to these objects.

In section \ref{examples} we point to some further examples of geometrical interpretations of generalized quivers; though the paper itself deals with complex Lie groups, in this section we point to an example encompassing \emph{real} reductive Lie groups. In section \ref{hk} we outline the proof a refinement of the Hitchin-Kobayashi correspondence for a K\"ahler fibration, which we use throughout the paper. This refinement allows for a systematic interpretation of parameters that show up in various gauge/vortex equations.

\subsection{Acknowledgments} I want first of all to thank my adviser, Peter Gothen, for his earnest guidance. I also owe a thankful note to Ignasi Mundet i Riera for hosting me in Barcelona, and generously granting me time for discussions; also, for first pointing Popovici's article to me. I have been supported by Fundação para a Ciência e a Tecnologia, IP (FCT) under the grant SFRH/BD/89423/2012, by Funda\c c\~ ao Calouste Gulbenkian under the program Est\'imulo \`a Investigação, and by CMUP (UID/MAT/00144/2013), which is funded by FCT (Portugal) with national (MEC) and European structural funds through the programs FEDER, under the partnership agreement PT2020.

\section{The Hitchin-Kobayashi Correspondence for Twisted Quiver Bundles}\label{plain}

\subsection{Preliminaries} We quickly review plain quiver representations, framing it in a way that will better serve our purposes. Let, then, $Q$ be a quiver, $X$ a compact K\"ahler manifold, $\mathbb{M}_\alpha$ twisting bundles, and $(\mathbb{V}_i, \phi_\alpha)$ be a twisted representation of $Q$ in the category of holomorphic vector bundles over $X$. Unless otherwise stated, we will always assume our objects to be holomorphic. Let $E^\mathbb{C}_i$ denote the frame bundle of $\mathbb{V}_i$, so that $\mathbb{V}_i=E^\mathbb{C}_i\times_\rho V_i$, where $V_i$ is the fibre space of the vector bundle and $\rho$ is the natural action of its automorphism group $\GL_i:=\GL(V_i)$. Then, $E^\mathbb{C}=\prod E^\mathbb{C}_i$ is a principal $\GL(\underline{i}):=\prod \GL_i$ bundle, and denoting $\mathbb{V}=\bigoplus \mathbb{V}_i$ and $V=\bigoplus V_i$, we have $\mathbb{V}=E^\mathbb{C}\times_\rho V$, where by an abuse of notation we also denote by $\rho$ the product action of $\GL(\underline{i})$ on $V$. We refer indifferently to $E^\mathbb{C}$, $V$ and $\mathbb{V}$ as the \emph{total space of the representation}.

Conversely, given a $\GL(\underline{i})$-principal bundle $E^\mathbb{C}$, it splits into factors $E^\mathbb{C}=\prod E^\mathbb{C}_i$ of principal $\GL_i$ bundles. Then, if we additionally have vector spaces $V_i$ with the right dimensions, the fibered products $\mathbb{V}_i$ correspond to a representation of the vertices of the quiver.

In the same way, for the twisting bundles, we get a correspondence between the vector bundles $\mathbb{M}_\alpha$, on one hand, and $\GL_\alpha$ principal bundles $F_\alpha$ together with vector spaces $M_\alpha$, on the other. We denote by  $F^\mathbb{C}$, $\mathbb{M}$ and $M$ the \emph{total twisting space}.

Consider now a morphism $\phi_\alpha : \mathbb{V}_{\ta} \otimes \mathbb{M}_\alpha \to \mathbb{V}_{\ha}$. This is a global section of the bundle of twisted homomorphisms, i.e., $\phi_\alpha \in \Omega^0 \left( \mathfrak{Hom}(\mathbb{V}_{\ta}\otimes \mathbb{M}_\alpha, \mathbb{V}_{\ha}) \right)$. The frame bundle of this latter vector bundle is $E^\mathbb{C}_{\ta}\times E^\mathbb{C}_{\ha} \times F^\mathbb{C}_\alpha$ and the fibre is $\Hom (V_{\ta}\otimes M_\alpha, V_{\ha})$, where the group $\GL_{\ta} \times \GL_{\ha} \times \GL_\alpha$ acts by the appropriate conjugation. Given a total space $\mathbb{V}$ and twisting space $\mathbb{M}$, a representation is then determined by a choice of an element $\phi\in \mathscr{S}:=\Omega^0 \left( \mathfrak{Rep} (Q,\mathbb{V}) \right)$, where by definition
\begin{equation*}
\mathfrak{Rep} (Q,\mathbb{V}):=\bigoplus_\alpha \mathfrak{Hom}(\mathbb{V}_{\ta}\otimes \mathbb{M}_\alpha, \mathbb{V}_{\ha})
\end{equation*}
is the representation space of the morphisms. Let $\mathrm{Rep}(Q,V)$ be the finite-dimensional representation space of $Q$ determined by the $V_i$. We can now see that $\mathfrak{Rep}(Q,\mathbb{V})=(E^\mathbb{C}\times F^\mathbb{C})\times_\rho \mathrm{Rep}(Q,V)$, where the representation $\rho$ is
\begin{equation*}
\rho \left((g_i)_i\times (g_\alpha)_\alpha \right) (\phi_\alpha)_\alpha=(g_{\ha}\circ \phi_\alpha \circ (g_{\ta}\times g_\alpha)^{-1})_\alpha
\end{equation*}

All this can be summarized as follows.

\begin{principalrep}
Let $Q$ be a quiver. A twisting for $Q$ is equivalent to a principal $\GL (\underline{\alpha})$-bundle $F^\mathbb{C}$ together with vector spaces $M_\alpha$. A twisted representation of $Q$ with dimension vector is equivalent to the prescription of a principal $\GL(\underline{i})$-bundle $E^\mathbb{C}$ together with a finite-dimensional representation space $\mathrm{Rep}(Q,V)$, and a section $\phi \in \Omega^0 ((E^\mathbb{C}\times F^\mathbb{C})\times_\rho \mathrm{Rep}(Q,V))$.
\end{principalrep}

From our notation, it is already obvious that we are interested in passing to maximal compacts, by endowing our vector bundles with hermitian metrics. Indeed, a hermitian metric on $\mathbb{V}_i$ corresponds to a reduction $h_i \in \mathcal{A}^0 (E_i^\mathbb{C} (\GL_i/\U_i))$ of the structure group of $E_i^\mathbb{C}$ to the unitary group $\U_i:=\U(V_i)$. The product of all such reductions is a section $h\in \mathcal{A}^0 \left( \prod E_i^\mathbb{C}(\GL_i/U_i)\right)=\mathcal{A}^0 (E^\mathbb{C}(\GL(\underline{i}) /\U(\underline{i})))$, where $\U (\underline{i}):=\prod \U_i$. Clearly, $\U(\underline{i})=\GL(\underline{i})\cap \U (V)$ for a particular choice of unitary group of $V$, so this reduction actually corresponds to a choice of a hermitian metric on the total space $\mathbb{V}$ satisfying some extra conditions(namely, that it restrict to a hermitian metric on each vertex, and different vertices are orthogonal.) Denote the reduced bundle thus obtained by $E$, and, analogously for a choice of hermitian metrics on the twistings, the reduced twisting bundle by $F$. We will call a pair $(E, \phi)$ a unitary quiver bundle.

\subsection{Hermitian metrics and holomorphic structures}

We shall very quickly review some basic facts about bundles that we'll need to understand the approach to the classification problem of quiver bundles. This review is necessarily very synthetic, and for more details on this material, excellent references are \cite{gh} or \cite{kobayashi}.

Let $\mathbb{V}$ now be a smooth (complex) vector bundle over $X$. Recall that a \emph{holomorphic structure} on $\mathbb{V}$ is a differential operator $\bar{\partial}_A:\mathcal{A}^0 (\mathbb{V}) \to \mathcal{A}^1 (\mathbb{V})$ obeying the Leibniz rule $\bar{\partial}_A (f\psi)=\bar{\partial}f \otimes s + f\bar{\partial}_A s$. This differential operator can be extended to higher order forms in the usual way, and a holomorphic structure is integrable if $\bar{\partial}_A^2=0$. Endowing $\mathbb{V}$ with an integrable holomorphic structure is equivalent to turning it into a holomorphic vector bundle $\mathbb{V}_A$ (hence the name,) and in fact $\Omega^0(\mathbb{V}_A)=\ker \bar{\partial}$. Just as in the case of connections, the space of holomorphic structure is an affine space modelled on the holomorphic one-forms.

It is a standard fact that in the presence of a hermitian metric, a choice of holomorphic structure $\bar{\partial}_A$ is equivalent to the choice of a unique smooth connection $\nabla_A$ on $\mathbb{V}$. In the direction relevant to us, the connection is uniquely determined by two requirements: that the holomorphic structure (in the sense of an anti-holomorphic differential operator) equal the $(0,1)$-part of the connection, and that the metric be parallel as a (smooth) section of $T^*E\otimes \overline{T^*E}$. A perhaps more elementary statement of the last requirement is that the connection satisfy
\begin{equation}\label{metricconnection}
dh(v,w)=h(\nabla_A v,w)+h(v,\nabla_A w)
\end{equation}
for every pair of sections $v,w\in \mathcal{A}^0(E)$, where $h$ is the metric. Connections satisfying condition (\ref{metricconnection}) are said to be metric, and the unique metric connection determined by the holomorphic structure of $\mathbb{V}_A$ is called the Chern connection of the metric. Then, the outcome of all this is the following: a pair $(\mathbb{V}_A, h)$ of a holomorphic bundle together with a hermitian metric is equivalent to a pair $(\mathbb{V}, \nabla_A, h)$ where $\mathbb{V}$ is the underlying smooth bundle, and $\nabla_A$ the (smooth) Chern connection. An important fact is that the curvature $F_A$ of $\nabla_A$ is always of type $(1,1)$ (and conversely, every such connection determines a holomorphic structure.)

Let $\mathscr{G}^\mathbb{C}$ be the complex gauge group of $\mathbb{V}$, i.e., the group of linear automorphisms $g:\mathbb{V} \to \mathbb{V}$ covering the identity on $X$. There is an obvious action of $\mathscr{G}^\mathbb{C}$ on connections, namely, $g\cdot \nabla_A= g\circ \nabla_A \circ g^{-1}$. This action does not respect the metric requirement, but after fixing a hermitian metric, we may (and do) restrict to the group of unitary gauge transformations $\mathscr{G}$, acting on the space $\mathscr{A}^{1,1}$ of metric connections. Then, the correspondence above is equivariant in the sense that if two connections are unitarily equivalent, then the induced holomorphic vector bundles are isomorphic. In the presence of a hermitian structure, then, we can then alternatively see the gauge group as acting as automorphisms of the holomorphic vector bundles $\mathbb{V}_A$, or by `moving around' the metric connections. \footnote{There's a third point of view, which in fact is the most common (and perhaps most correct) in these classification problems. One could see the group of isomorphisms of holomorphic vector bundles as `moving around' hermitian metrics, rather than their connections. The Hitchin-Kobayashi correspondence, then, would be seen as a theorem on thhe hermitian structures of a holomorphic vector bundle.}

\subsection{The classification problem}

The general task of the theory of quiver bundles is the following:
\begin{classification1}
Classify quiver bundles representations with fixed underlying smooth total space, up to isomorphism.
\end{classification1}

We will often blur the distinction between a holomorphic bundle and its underlying smooth bundle, using the same notation when the context makes clear which one we mean. This restriction to a fixed underlying total space is not as restrictive as it might seem at first sight, since the smooth classification is also a \emph{topological} one. We have just seen that this problem is equivalent to

\begin{classification2}
Classify pairs $(A, \phi) \in \mathscr{A}^{1,1} \times \mathscr{S}$ up to unitary gauge transformations.
\end{classification2}

Using this correspondence between a representation with reduction and pairs of a hermitian connection with a section, the Hitchin-Kobayashi correspondence is established in terms of the latter. This correspondence establishes the equivalence between an algebraic criterion of stability, and the solutions of a certain gauge-theoretical partial differential equation. The proof of such correspondence is rather technical and involved, so we delay it to section \ref{hk}. However, the case at hand is both well known, and sufficiently illustrative of the general case, so we proceed toward the statement of the correspondence. We note that our treatment of the correspondence is the one due to Mundet i Riera in \cite{mundet}; for details, one should refer back to that article.

\subsection{The gauge equations}

The first element in a Hitchin-Kobayashi correspondence are the gauge equations, a set of PDEs envolving the Chern connection of the hermitian structure. The spaces $\mathscr{A}^{1,1}$ of connections of type $(1,1)$ on $E\times F$, and on the space of representations of morphisms $\mathscr{S}:=\Omega^0(\mathfrak{Rep}(Q,V))$ turn out to admit K\"ahler structures for which the action of the gauge group admits a moment map $\mu_{\mathscr{A}\times \mathscr{S}}: \mathscr{A}^{1,1}\times \mathscr{S} \to (\mathrm{Lie}\phantom{.}\mathscr{G})^*$. The gauge equations are then $\mu_{\mathscr{A}\times \mathscr{S}}=c$, where $c$ is a central element of $(\mathrm{Lie}\phantom{.}\mathscr{G})^*$.

Recall that we have an embedding $U(\underline{i})\to \U (V)$, which induces a $\mathrm{Ad}\phantom{.}\mathfrak{u}(\underline{i})$-equivariant isomorphism $\mathfrak{u}(\underline{i})\simeq \mathfrak{u}(\underline{i})^*$. We will quickly review the construction of the moment map for the action of the gauge group; for details, cf. \cite{mundet}.

The space $\mathscr{A}^{1,1}$ of smooth connections on $\mathbb{V}$ is an affine space modelled on $\mathcal{A}^1(E(\mathfrak{u}(\underline{i})))$. Since the almost complex structure on $X$ induces an almost complex structure on $T^*X$, it also induces one on $\mathscr{A}$. If we denote by $\Lambda : \Omega^* \to \Omega^{*-2}$ the K\"ahler endomorphism, i.e., the adjoint of wedging by the K\"ahler form on $X$, then we can define
\begin{equation*}
\omega_\mathscr{A} (A,B)=\int_X \Lambda W(A,B)
\end{equation*}
where $W$ is the wedge product followed by the pairing on $\mathfrak{k}$ induced by the faithful representation. Then, $\omega_\mathscr{A}$ is a symplectic form, the \emph{Atiyah-Bott symplectic form}. The action of the gauge group preserves this symplectic form, and indeed there is a moment map
\begin{equation*}
\mu_\mathscr{A}(A)=\Lambda F_A
\end{equation*}
where $F$ is the curvature of $A$. Since $\mathscr{A}^{1,1}$ is a complex $\mathscr{G}$-invariant subspace of $\mathscr{A}$ (recall that holomorphic structures are modelled on $\mathscr{A}$,) all this applies to it as well.

According to the splitting $E=\prod E_i$ of the $\GL (\underline{i})$-bundle into a product of $\GL_i$-bundles, each $\GL (\underline{i})$ connection splits into a sum $A=\oplus A_i$. In turn, this means that the space $\mathscr{A}$ of connections splits as $\mathscr{A}=\prod \mathscr{A}_i$. The point of this observation is the following: given two arbitrary symplectic spaces $(X_i,\omega_i)$, $i=1,2$, and two positive numbers $a_i$, the form $a_1 \omega_1+a_2 \omega_2$ makes the product $X_1\times X_2$ a symplectic space. Then, rather than taking the Atiyah-Bott form for the product group, we can take a weighted sum of the forms for the factor groups. Choose positive real numbers $a_i$, and take them as weights for the sum of symplectic forms. Then, the component along $\mathfrak{gl}_i:=\mathrm{Lie}\phantom{.} \GL_i$ of the moment map for the action of the gauge group is then
\begin{equation}\label{connectionmoment}
\mu_i (A)= a_i \Lambda F_i
\end{equation}

Since the action is trivial on the bundle $M$, the moment map for the space of its connections can be taken to be zero. Hence, we can simply extend the map defined by (\ref{connectionmoment}) by zero along $\mathfrak{gl}_\alpha$ to get a moment map on $\mathscr{A}^{1,1}$.

Our fibre space, on the other hand, is a finite-dimensional complex vector space (namely, the finite-dimensional representation space of the quiver $Q$.) In fact, it has an induced hermitian metric as soon as we choose one on each vertex of the representation. Since the action of $K$ is unitary, there is then an induced moment map for the action. In general, as explained in \cite{king}, for a vector space with hermitian form, the moment map for a unitary action is $\langle \mu (x), A \rangle= (\mathscr{X}_{iA}x, x)$, where $\mathscr{X}_A$ is the vector field induced by the action of $A$. In our case, given hermitian metrics on $V_{\ta}$, $V_{\ha}$ and $M_\alpha$, there is an induced hermitian form on $\mathfrak{gl}_{\ta}\times \mathfrak{gl}_{ha} \times \mathfrak{gl}_\alpha=\Hom (V_{\ta}\times M_\alpha, V_{\ha})$, given by $(\phi, \psi)=\tr (\phi^* \psi)$. An explanation is due about the dual of $\phi_\alpha$: formally speaking, with respect to the hermitian metrics on $\mathbb{V}$ and $\mathbb{M}$, we have $\phi_\alpha \in \Hom (\mathbb{V}_{\ha}, \mathbb{V}_{\ta}\otimes \mathbb{M}_\alpha)$, so there is no trouble with the composition $\phi_\alpha \circ \phi_\alpha^*$. For the other composition, recall that $\Hom (\mathbb{V}_{\ta}\otimes \mathbb{M}_\alpha, \mathbb{V}_{\ha})=\Hom (\mathbb{V}_{\ta}, \mathbb{M}_\alpha^* \otimes \mathbb{V}_{\ha})$. A straightforward computation gives for the component of the moment map on the fibers along $\mathfrak{gl}_i$ the expression
\begin{equation*}
-i \mu_i (\phi)=\sum_{\ha=i} \phi_\alpha \phi_\alpha^* - \sum_{\ta=i}\phi_\alpha^* \phi_\alpha =: [\phi, \phi^*]
\end{equation*}
where $i=\sqrt{-1}$, an unfortunate duplication of the letter which nonetheless should not cause confusion since we never use the indices of the vertices as constants. (In \cite{king} no such factor of $-i$ appears explicitly because the moment map already goes to $(i\mathfrak{k})^*$, and not to $\mathfrak{k}^*$.) 

This construction endows each fibre of the bundle $\mathfrak{Rep}(Q;V)$ with a symplectic form $\omega$, and this extends to a symplectic form on the space of sections $\mathscr{S}$ by the formula
\begin{equation*}
\omega_\mathscr{S} (s,v)=\int_X \omega (s(x),v(x))
\end{equation*}
Here, $s, v \in \mathscr{S}$. It follows that the action of the gauge group is symplectic, and that there is a moment map $\mu$ which is fibrewise the moment map of the K\"ahler fibre.

Finally, a central element in the Lie algebra $\mathfrak{gl}_i$ is just a constant multiple of the identity $i\tau_i \mathrm{id}_i$, with $\tau_i$ a real number. Thus, the components of the gauge equation are
\begin{equation}
i a_i \Lambda F_i + [\phi, \phi^*] = \tau_i \mathrm{id}_i
\end{equation}
Note that if we added a parameter for the moment map on the sections, it could just simply be absorbed by the two other parameters.

\subsection{Parabolic subgroups}

Parabolic subgroups play an important role in the abstract Hitchin-Kobayashi correspondence; the references for this material are \cite{mundet} and \cite{ggm}. Let $G$ be a connected complex reductive Lie group, $K$ a maximal compact, $\mathfrak{g}$ and $\mathfrak{k}$ the corresponding Lie algebras. As above, suppose a faithful  representation $\rho: K \to \U(V)$ is given, along with an induced, implicit isomorphism $\mathfrak{k}\simeq \mathfrak{k}^*$.

If $\mathfrak{z}$ is the centre of $\mathfrak{g}$ and $T$ a maximal torus of $K$, there is a choice of Cartan subalgebra $\mathfrak{h}$ such that $\mathfrak{z}\oplus \mathfrak{h}=\mathfrak{t}^\mathbb{C}$, where $\mathfrak{t}=\mathrm{Lie}\phantom{.}T$. Let $\Delta$ be a set of roots, and $\Delta '$ be a choice of simple roots for the Cartan decomposition of $\mathfrak{g}$ with respect to $\mathfrak{h}$. For any subset $A=\{ \alpha_{i_1},...,\alpha_{i_s}\} \subset \Delta '$, define
\begin{equation*}
D_A=\{ \alpha \in R | \alpha =\sum m_j\alpha_j, \textrm{ $m_{i_t}\geq 0$ for $1\leq t\leq s$} \}
\end{equation*}

The \emph{parabolic subalgebra} associated to $A$ is $\mathfrak{p}=\mathfrak{z}\oplus \mathfrak{h}\oplus \bigoplus_{\alpha\in D_A} \mathfrak{g}_\alpha$; the subgroup $P$ of $G$ determined by this subalgebra is the \emph{parabolic subgroup} determined by $A$. A \emph{dominant (resp. antidominant) character} of $P$ is a positive (resp. negative) combination of the fundamental weights $\lambda_{i_1},...,\lambda_{i_s}$ plus an element of the dual of $i(\mathfrak{z}\cap \mathfrak{k})$.

Due to our choice of Cartan subalgebra, and using our implicit isomorphism $\mathfrak{k} \simeq \mathfrak{k}^*$, an anti-dominant character $\chi$ of $P$, may be identified with an element of $i\mathfrak{k}$; we'll still denote by $\chi$. We have that $\rho(\chi)$ is hermitian (since $\chi \in i\mathfrak{k}$,) it has real eigenvalues $\lambda_1< ... <\lambda_j<...<\lambda_r$, and it diagonalizes. In other words, $\chi$ induces a filtration
\begin{equation*}
0 \neq V^1\varsubsetneq ... \varsubsetneq V^r =V
\end{equation*}
where $V^k=\bigoplus_{i\leq k}V_{\lambda_k}$ is the sum of all eigenspaces $V_{\lambda_i}$ with $i\leq k$. The following theorem is from \cite{mundet}, section 2:
\begin{flagcharacter}
Let $\chi \in i\mathfrak{k}$. Then, the pre-image by $\rho$ of the stabilizer of the induced flag is a parabolic group $P(\chi)$, and $\chi$ is the dual of an antidominant character of $P(\chi)$. Further, given an arbitrary parabolic subgroup $P$ with Lie algebra $\mathfrak{p}$, there is a choice of a Cartan subalgebra $\mathfrak{h}\subset \mathfrak{g}$ contained in $\mathfrak{p}$ such that $\chi \in \mathfrak{h}$ and is antidominant with respect to $P$ if and only if $P$ stabilizes the flag induced by $\chi$.
\end{flagcharacter}

It is important, however, to keep in mind that this applies to filtrations induced by an element of $i\mathfrak{k}$, and not just any filtration of $V$. For $\GL (V)$, it is true that any flag is induced by such an element. For an orthogonal or symplectic group, however, the flags induced in this way are \emph{isotropic} flags, i.e., flags where $V^{r-k}=(V^{k})^\perp$ (so that for $k\leq r/2$, the spaces are actually isotropic.)

Again following Mundet in \cite{mundet}, if we additionally are given a holomorphic reduction $\pi: X\to E(G/P)$ to $P$, the pair $(\pi, \chi)$ determines an element $g_{\pi, \chi}\in \Omega^0 (E\times_\mathrm{Ad} i\mathfrak{k})$ which is fibrewise the dual of $\chi$. Considered as an endomorphism of $\mathbb{V}:= E\times_\rho V$ (through the representation $\rho$,) it has almost-constant eigenvalues $\lambda_1< ... <\lambda_k<...<\lambda_r$,  and so induces a filtration
\begin{equation*}
0\subset \mathbb{V}^1 \subset ... \subset \mathbb{V}^k \subset ... \subset \mathbb{V}^r=\mathbb{V}
\end{equation*}
which is defined outside of a codimension-two submanifold. Here, $\mathbb{V}^{\lambda_k}=\bigoplus_{i\leq j} \mathbb{V}(\lambda_i)$ is the sum of all eigenbundles with $\lambda_i\leq \lambda_k$. Again, from \cite{mundet}, we have

\begin{bundlefiltration}
If the reduction is holomorphic, for any antidominant character $\chi$, the induced filtration is holomorphic. Conversely, an element $g\in \Omega^0 (E\times_{\mathrm{Ad}} i\mathfrak{k})$ with constant eigenvalues determines a holomorphic reduction $\pi$ and an antidominant character $\chi$ such that $g=g_{\pi, \chi}$.
\end{bundlefiltration}

\subsection{Stability}

The second ingredient for the Hitchin-Kobayashi corrrespondence is the total degree. This is an numerical parameter that in the K\"ahler setting plays the role analogous to that of the Hilbert-Mumford criterion in GIT. In fact, we shall see that the Hitchin-Kobayashi correspondence characterizes the level sets of the moment map (which are given by the gauge equations) by the behaviour of this parameter.

Let $K$ be a compact Lie group, $G$ its complexification, and $E$ a principal $K$-bundle, and suppose that  $K$ acts by symplectomorphisms on a K\"ahler manifold $F$, and that a moment map $\mu$ for this action exists. For any $x\in F$, $k\in \mathfrak{k}$, and $t\in \mathbb{R}$ let
\begin{equation*}
\lambda_t (x, k)=\langle \mu (\exp (itk)x), k\rangle
\end{equation*}
where $\langle \cdot, \cdot \rangle$ is the canonical pairing of $\mathfrak{k}$ with its dual. Then, the \emph{maximal weight} of the action of $k$ on $x$ is
\begin{equation*}
\lambda (x,k)=\lim_{t\to \infty} \lambda_t (x,k)
\end{equation*}
In \cite{mundet}, it is explicitly shown how this numerical parameter gives a linear criterion for the stability of a point in the K\"ahler manifold.

Now, for a given pair $(\pi, \chi)$ of a reduction to a parabolic subgroup $P\subset G$ plus an antidominant character of $P$, as above, define
\begin{equation*}
\deg (\pi, \chi)=\lambda_r \deg (\mathbb{V}^r)+ \sum_{k=1}^{r-1} (\lambda_k - \lambda_{k+1})\deg (\mathbb{V}^k)
\end{equation*}
where $\mathbb{V}^k$ are the bundles in the induced filtration, and $\deg (\mathbb{V})$ is the degree of the vector bundle $\mathbb{V}$. This is an infinite-dimensional version of the maximal weight defined above, if we consider the Atiyah-Bott symplectic form on the space of connections on a principal $K$ bundle $E$.

Since our structure group is actually a product of groups, and our connections split accordingly, nothing demands that we take the Atiyah-Bott forms for the total bundle directly. In fact, just above we took a weighted sum of the Atiyah-Bott forms on the vertices. A reduction $\pi$ induces a reduction $\pi_i$ on each of the vertices, and an anti-dominant character obviously splits $\chi=\oplus \chi_i$, like the Lie algebra. But then, the maximal weight changes accordingly, and in fact we should consider instead
\begin{equation*}
\deg_a (\pi, \chi)= \sum a_i \deg (\pi_i, \chi_i)
\end{equation*}
Here, the positive numbers $a_i$ are the parameters for the moment map, as above. In other words, for a weighted sum of moment maps, we must take a weighted sum of degrees. We'll call this the `$a$-degree.'

Finally, given a section $\phi \in \Omega^0 (E(F))$ of the associated bundle with fibre $F$, and a central element $c\in \mathfrak{k}$, the total $c$-degree of $(\sigma, \chi)$ is defined as
\begin{equation*}
T_\Phi^c(\pi, \chi)=\deg_{a} (\pi, \chi) + \int_X \lambda (\phi(x), -ig_{\pi, \chi}(x))+\langle i\chi, c\rangle \mathrm{Vol} (X)
\end{equation*}
Here, $g_{\pi, \chi}\in \Omega^0 (E\times_\mathrm{Ad}i\mathfrak{k})$ is the fibrewise dual of $\chi$. The total degree is allowed to be $\infty$.

We need to compute all these parameters in the specific case of quiver representations. Since for this case the K\"ahler fibre is actually a vector space, $g_{\pi, \chi}(x)$ can be seen as an endomorphism of $\mathrm{Rep}(Q,V)$, and it makes sense to speak of its eigenvalues. Let $\mathcal{F}^- (\chi) \subset E(F)$ be the subset of vectors where $g_{\pi, \chi}$ acts negatively, i.e., vectors in the direct sum of all the negative eigenspaces of $g_{\pi, \chi}$. Since $g_{\pi, \chi}$ has constant eigenvalues, and $\pi$ is holomorphic, $\mathcal{F}^- (\chi)$ is a holomorphic subbundle. Then, one computes that
\begin{equation}\label{sectionweight}
\lambda (\phi(x), g_{\pi, \chi}(x))=\left\{
\begin{array}{ll}
0 & \mathrm{if} \quad \phi(x) \in \mathcal{F}^- (\chi)\\
\infty & \mathrm{if } \quad \phi(x) \notin \mathcal{F}^- (\chi)
\end{array}
\right.
\end{equation}

This result can be intuitively understood if we just look at the action of a group of matrices on a vector space. Then, a one-parameter subroup corresponds to repeated application of the same automorphism to a vector. The action on the subgroup on an eigenspace is then $t^\lambda$, where $\lambda$ is an eigenvalue of the generator. As long as all $\lambda$ are negative, this tends to zero; a positive one forces a divergence to infinity.

On the other hand, $c=\oplus i\tau_i \mathrm{id}_i$, and $\chi=\oplus \chi_i$ for $\chi_i \in \mathfrak{gl}_i$, so
\begin{equation*}
\langle i\chi, c\rangle = -\sum \tau_i \langle \chi_i , \mathrm{id}_i \rangle = -\sum \tau_i \tr \chi_i
\end{equation*}

From here on out, we assume that the volume of $X$ has been normalized to one. What we've seen so far justifies the following definition:

\begin{cst}
A pair $(A,\phi) \in \mathscr{A}^{1,1}\times \Omega^0 (\mathfrak{Rep})$ is $(a_i, \tau_i)$-stable if for any $(\pi, \chi)$ with $\phi(X)\subset \mathcal{F}^- (\chi)$ we have
\begin{equation*}
\sum \left( a_i \deg (\pi, \chi_i) - \tau_i \tr \chi_i \right)>0 
\end{equation*}
\end{cst}

Given a subrepresentation $\mathbb{V}'\subset \mathbb{V}$, consider the one-term flag $0\subset \mathbb{V}'\subset \mathbb{V}$. Since we're dealing with a general linear group, the subgroup $P$ fixing this flag is a parabolic subgroup, and there is an anti-dominant character $\chi$ of $P$ inducing that flag. For a particular choice of $\chi$, we have $\deg (\sigma_i, \chi_i)=\deg (\mathbb{V}_i) - \deg (\mathbb{V}'_i)$, and $\tr \chi_i=\rk \mathbb{V}_i - \rk \mathbb{V}'_i$, and we get the more familiar definition in terms of slopes. In this case, then, the total degree can be interpreted as the maximal weight for the moment maps we constructed. In other words, the slope condition for subrepresentations is a necessary condition for stability. In general, however, parabolic subgroups stabilize more complicated flags, and although our definition looks complicated, it generalizes to other reductive groups. Another important point is that we can restric only to subrepresentation when working over a Riemann surface, for only over a surface can we keep inside the category of vector bundles over $X$. From this point of view, what is surprising is that our definition even works. The essential ingredient is the remarkable theorem of Uhlenbeck and Yau \cite{uy}, which assures us that the subsheaves in the filtration define vector subbundles over a submanifold of codimension at least two. This is ultimately reliant on a Hartog-like extension theorem, and the interested reader can look up Popovici's nicely geometric proof of this fact in \cite{popovici}.

Essentially, we can't always interpret the total degree as a maximal weight because we're working in the wrong category: if we recall that every coherent subsheaf defined outside of a codimension two sumbanifold has a unique extension as a coherent subsheaf, we see that for a general $\chi$, we get a filtration of $\mathbb{V}$ by coherent subsheaves, and not just subbundles (in fact, we're reversing the whole story, since what one can prove is that the filtration is already by coherent subsheaves.) We have the following:

\begin{morphismsfactor}
In the conditions of the definition of stability above, the morphisms restrict to each element in the filtration, i.e., if $\phi \in \mathcal{F}^-(X)$, then $\phi (\mathbb{V}_{k,\ta} \otimes M_\alpha)\subset \mathbb{V}_{k, \ha}$ for all $k$ in the filtration.
\end{morphismsfactor}

\begin{proof}
Given any morhpism $\phi_\alpha$ in the representation; we can consider it as an element of $\mathrm{Rep}(Q, V)$ by taking all the other components to be zero. Since $\phi (X)\in \mathcal{F}^-(\chi)$, we may assume that $\phi_\alpha$ is an eigenvector of $\chi$ with eigenvalue $\lambda<0$, i.e., $[ \chi, \alpha ]=\lambda \alpha$.

Now, the character $\chi$ also acts on the total space $V$, so let $x\in V_{\ta}$ be an eigenvector with eigenvalue $\lambda_x$. We have
\begin{equation*}
\chi (\alpha (x))=\alpha (\chi (x)) + \lambda \alpha (x)= (\lambda_x+\lambda)\alpha (x)
\end{equation*}
Since $\lambda<0$, $\lambda_x + \lambda<\lambda_x$, as desired. The general case follows from here.
\end{proof}

The following definition, to be found in \cite{ag}, should be well-motivated by the previous observations:

\begin{quiversheaf}
A quiver sheaf representation $(\mathcal{E}, \phi)$ is a collection of coherent sheaves $\mathcal{E}_i$, one for each vertex of the quiver, together with a collection of sheaf morphisms $\phi_\alpha : \mathcal{E}_{\ta}\otimes \mathcal{M}_\alpha\to \mathcal{E}_{\ha}$, one for each arrow in the quiver (the $\mathcal{M}_\alpha$ are the twisting sheaves.)
\end{quiversheaf}

 Recall that the degree of a coherent sheaf is defined as
\begin{equation*}
\deg (\mathcal{E})=\frac{1}{(n-1)!}\frac{2\pi}{\mathrm{Vol}X} \langle c_1 (\mathcal{E})\cup [\omega^{n-1}], [X] \rangle 
\end{equation*}
where $c_1(\mathcal{E})$ is the first Chern class of $\mathcal{E}$, $[\omega^{n-1}]$ is the class of the K\"ahler form of $X$, and $[X]$ its fundamental class. Given a quiver sheaf, for $\sigma$ and $\tau$ be collections of real numbers $\sigma_i$, $\tau_i$ with $\sigma_i >0$, we define the degree and slope of the representation respectively to be
\begin{equation*}
\deg_{a, \tau}(\mathcal{E})=\sum (a_i\deg (\mathcal{E}_i - \tau_i \rk (\mathcal{E}_i))
\qquad \mu_{a, \tau} (\mathcal{E})=\frac{\deg_{a, \tau}(\mathcal{E})}{\sum a_i \rk (\mathcal{E}_i)}
\end{equation*}

Note that our situation fits into this framework since we can always take the sheaves of sections of our vector bundles. This inclusion in fact respects the stability of the representation in the following sense: 

\begin{simplifiedplainstability}
A representation $(E, \phi)$ is stable (as a bundle representation) if and only if $\mu_{a, \tau}(\mathcal{F})<\mu_{a, \tau}(E)$ for every proper sub-sheaf representation $0\neq \mathcal{F} \subset E$.
\end{simplifiedplainstability}

The main idea of the proof is that for our group, one can reduce the case of a general filtration to the filtration by one subsheaf (see, e.g., \cite{ggm} for the special case of Higgs bundles.)

\subsection{The correspondence}

We need a technical, but important definition:
\begin{simple}
A pair $(A, \phi)$ is infinitesimally simple if no non-central semisimple element in the Lie algebra of the complex gauge group fixes $(A, \phi)$.
\end{simple}

In our case, this condition is equivalent to the following: suppose that $E\times_\rho \mathrm{Rep}(Q,V)$ splits non-trivially as a direct sum, and that there is a reduction of the structure group of $E$ to $G'\subset \GL(\underline{i})$ compatible with the splitting. Then $(A,\phi)$ is infinitesimally simple if there is no element in the Lie algebra of $G'$ which both annihilates the section $\phi$ \emph{and} acts non-trivially on $E\times_\rho \mathrm{Rep}(Q,V)$. For details, cf. \cite{banfield} and \cite{mundet}.

The Hitchin-Kobayshi correspondence of section \ref{hk} gives the following:

\begin{hksimplequiverreps}
Let $(A, \phi)$ be an infinitesimally simple pair. Then, $(A, \phi)$ is $(a_i, \tau_i)$-stable if and only if it is gauge equivalent to a pair satisfying the gauge equations with parameters $(a_i, \tau_i)$.
\end{hksimplequiverreps}

This correspondence has been entirely put in terms of pairs. In terms of representations of $Q$, one would say that for an infinitesimally simple representation (which here means that its only automorphisms are multiples of the identity,) there is a hermitian metric with Chern connection satisfying the gauge equations if and only if it is $(a_i, \tau_i)$-stable. This is the more traditional point of view for this type of correspondence, by taking the gauge equations to be a condition on the metric structure. Our point of view, however, is better adapted to uncover the geometric significance of the correspondence in the K\"ahler case.

\subsection{General polystability}
We make a detour here to discuss polystability of quiver bundles. None of our remarks here extend to our later sections, so we only include them for completeness. Also for this reason, our remarks necessarily consist of generalities grazing only the surface of this topic. The main point is that we can still get an explicit description of representations satisfying the gauge equations even if we don't restrict attention to infinitesimally simple pairs. Such a description is essential for the study of the moduli space of representations, since the moduli of stable pairs is generally not compact.

As it happens, the category of plain quiver sheaves associated with a given quiver forms an abelian category. As is well known (e.g., in the case of vector bundles,) the existence of the abelian structure significantly simplifies the description of polystability. First of all, the study of polystability implies reductions to Levi subgroups of the parabolic groups in question. We've seen above that the parabolic groups correspond naturally to certain filtrations of the representation; the reduction to a Levi subgroup, in the context of an abelian category, corresponds to taking the associated graded object. The second fundamental fact is the existence of a \emph{Jordan-H\"older filtration}, which implies that every representation can lead to such a graded object. This latter filtration is a filtration of the form
\begin{equation*}
0=F_0\subset F_1\subset ... \subset F_k=F
\end{equation*}
where each consecutive quotient $F_i/F_{i-1}$ is stable. For a polystable (more generally, semistable) object, one can show that any two such filtrations have the same length, and yield the same graded object (though in general the filtrations themselves are \emph{not} isomorphic.) A very elucidating example is the finite dimensional case, cf. \cite{king}. The outcome of such considerations is that a representation is polystable if it is a direct sum of stable representations.

To state the correspondence as it appears in the literature, recall the definitions of $(\sigma, \tau)$ degree and slope. The definition and theorem that follow come straight from \cite{ag}.

\begin{psreps}
Let $(\mathcal{E}, \phi)$ be a quiver sheaf. Then, $\mathcal{E}$ is stable (resp. semistable) if for all proper quivers subsheaves $\mathcal{F}$ we have $ \mu_{\sigma, \tau} (\mathcal{F}) <  \mu_{\sigma, \tau} (\mathcal{E})$ (resp., $ \mu_{\sigma, \tau} (\mathcal{F}) \leq \mu_{\sigma, \tau} (\mathcal{E})$); $\mathcal{E}$ is polystable if it is a direct sum of polystable quivers sheaves, all of them with the same slope.
\end{psreps}

\begin{hkpsquivers}
Let $(\mathcal{E}, \phi)$ be a holomorphic twisted quiver bundle with $\deg_{\sigma, \tau} (\mathcal{E})=0$. Then, $\mathcal{E}$ is $(\sigma, \tau)$-polystable if and only if it admits a hermitian metric satisfying the $(\sigma, \tau)$-gauge equations. This hermitian metric is unique up to multiplication by a constant for each summand in the decomposition as a direct sum of stable subrepresentations.
\end{hkpsquivers}

This theorem is easily proven using our theorem above after one shows that any stable pair is infinitesimally simple (in fact simple.) We won't prove this here, however, since this is not relevant to our discussion below.

\section{Generalized Quivers}\label{generalized}

Derksen and Weyman \cite{derksen} introduced a concept of generalized quiver for a reductive algebraic group $G$ which we'll need for our generalization of quiver bundles. In this section, we review their definition, and introduce twistings for generalized quivers, which frequently appear in applications. Then, we introduce the notion of twisted generalized quiver bundles, and briefly discuss their Hitchin-Kobayashi correspondences.

\subsection{Generalized quivers} 
Let $G$ be a complex reductive Lie group, $\mathfrak{g}$ its Lie agebra. The definitions are as follows:
\begin{generalizedquiver}
\begin{enumerate}
\item A \emph{generalized $G$-quiver $\tilde{Q}$ with dimension vector} is a triple $(H,R,\mathrm{Rep}(\tilde{Q}))$ where $H$ is a closed abelian reductive subgroup of $G$, $R$ its centralizer in $G$, and $\mathrm{Rep}(\tilde{Q})$ a finite-dimensional representation of $R$ (the \emph{representation space}.) We require the irreducible factors of the representation also to be irreducible factors of $\mathfrak{g}$ as an $\mathrm{Ad} R$-module, and the trivial representation to not occur.

\item A representation of $\tilde{Q}$ is a vector $\phi \in \mathrm{Rep}(\tilde{Q},V)$.
\end{enumerate}
\end{generalizedquiver}

The motivation for this definition comes from considering the case $V=\bigoplus V_i$, $G:=\GL (V)$, and $H=\left\{ \prod \lambda_i \mathrm{id}_i | \lambda_i \in \mathbb{C}^* \right\}$. Then, the centralizer of $H$ is $R=\GL (\underline{i})$, and under the adjoint action of $R$, we have a decomposition $\mathfrak{g}=\bigoplus \Hom (V_i,V_j)$. Therefore, a choice of a representation space consists of picking a bunch of the $\Hom (V_i,V_j)$ with possible repetitions. We recover a classical quiver by drawing one vertex $i$ for each space $V_i$, and one arrow $\alpha : i\to j$ for each time the piece $\Hom (V_i,V_j)$ appears in the representation space -- this also justifies our notation $\mathrm{Rep}(\tilde{Q})$. Finally, an element $\phi \in \mathrm{Rep}(\tilde{Q})$ is precisely a representation of the classical quiver constructed in this way. In the next section, we shall see, following Derksen-Weyman, that there is an analogously concrete way of understanding $\Or (V)$- or $\Sp (V)$-generalized quivers.

Our definition is slightly different from Derksen-Weyman's. First of all, they define a representation to be the orbit $R\cdot \phi$, rather than the vector $\phi$ itself. In our context, this is clearly unsatisfactory. Also, they consider the case of algebraic groups, and correspondingly, $H$ is required to be Zarisky closed. We could have let $G$ be an arbitrary (possibly real) reductive Lie group, and indeed some important examples would be encompassed by such a definition (cf. \ref{examples},) but we'll throughout restrict to complex reductive groups as in the definition (so that, in practice, we've not strayed from Derksen-Weyman's definition, at least when the group is connected.)

Twistings are important when we come to quiver bundles, so we also need to generalize this notion. The orthogonal case in the next section will provide a very concrete motivation for our definition (see Remark \ref{remarktwisting} below.) Write $\mathrm{Rep}(\tilde{Q})=\oplus Z_\alpha$, where each $Z_\alpha$ is an irreducible subrepresentation (note that the factors might repeat with different labels.)

\begin{twistedgenquiver}
A twisting for a generalized $G$-quiver $\tilde{Q}=(H,R,\mathrm{Rep}(\tilde{Q},V))$ is a choice of a vector space $M_\alpha$ for each irreducible summand $Z_\alpha$. A finite-dimensional twisted representation of $\tilde{Q}$ is a element $\phi \in \mathrm{Rep}(\tilde{Q}, M):=\bigoplus Z_\alpha \otimes M_\alpha$.
\end{twistedgenquiver}

As in section \ref{plain}, $M_\alpha$ will for us be the fibre space of a vector bundle $\mathbb{M}_\alpha$ over $X$; we'll again denote by $\mathbb{M}$ the total twisting space, and by $F$ the respective principal bundle. We call $\mathrm{Rep}(\tilde{Q}, M)$ the space of $M$-twisted representations of $\tilde{Q}$.

\subsection{Generalized quiver bundles}
We take our clue from our study of quiver representations. When we considered classical quiver representations, we made an explicit distinction between the ``quiver-ly'' aspect (the fibre space), and the ``bundly'' part (the principal bundle.) It is only natural to make the following definition:

\begin{generalizedbundle}
Let $\tilde{Q}=(H,R,\mathrm{Rep}(\tilde{Q}))$ be a $G$-generalized quiver with twisting $\mathbb{M}_\alpha$, and $X$ a compact K\"ahler manifold. A twisted $\tilde{Q}$-bundle representation over $X$ is a pair $(E^\mathbb{C},\phi)$ of a principal $R$-bundle $E$ over $X$, and section $\phi \in \Omega^0 ((E^\mathbb{C}\times F^\mathbb{C})\times_\rho \mathrm{Rep}(\tilde{Q},M))$.
\end{generalizedbundle}

As usual, $F^\mathbb{C}$, the frame bundle of the total twisting space $\mathbb{M}$, is a $\GL (\underline{\alpha})=\prod \GL_\alpha$ bundle. A direct comparison with classical quivers should hint that our Hitchin-Kobayashi correspondence can be put to good use here as well. In fact, the whole analysis is technically the same as for the classical representations, except we must take a slightly more abstract point of view. The first concern is a reduction to a maximal compact subgroup of $R$, i.e., a `unitary representation': a representation together with a reduction of the structure group to a maximal compact $K$ of $R$. On the other hand, from the fact that the irreducible summands of our representation of $R$ are (complex) subspaces of the Lie algebra $\mathfrak{g}$, a moment map for the action can easily be constructed using the methods of section 6 of King \cite{king}. From here, we see that a Hitchin-Kobayashi correspondence exists for generalized quiver bundles, even without the correspondence we prove below. In fact, Mundet's theorem \cite{mundet} applies directly. However, as we shall see in the next sections, the structure group of a generalized quiver of interest (i.e., the maximal compact in the unitary representation) typically splits into a product. This means that we can introduce extra stability parameters in the correspondence, using the methods in section \ref{hk}.

As an aside, the splitting of the structure group is associated with the corresponding splitting in the reductive abelian group $H$. In fact, the identity component of $H$ is a torus $H_0=(\mathbb{C}^*)^r$, and $H$ splits as $H=(\mathbb{C}^*)^r \times F$, where $F$ is a finite group. It is an interesting question to determine under which conditions this splitting extends to the centralizer $R$, but alas it is a question we cannot answer. We'll content ourselves in the next sections to show how this works in important particular cases.

\section{Symmetric Quivers}\label{orthreps}
This section has a dual purpose. First of all, starting from a classical point of view, we want to consider a natural setting for the concept of orthogonal and symplectic symmetries of representations, the symmetric quivers. The orthogonal and symplectic structures are simple enough that we can deal with them directly, and, in fact, we prove a Hitchin-Kobayashi correspondence for orthogonal and symplectic by reference to the plain case. Now, from the generalized point of view, it makes sense to speak not of symmetric quivers, but of generalized quivers for orthogonal or symplectic groups. Our second goal with this section is to show that this instance of generalized quivers coincides with the more geometric concept of symmetric quivers. Derksen-Weyman \cite{derksen} established this result in finite dimensions, and after reviewing it, we prove its gauge-theoretic version. At the end of the section, we prove the infinite-dimensional one for quiver bundles. We remind the reader of our standing convention of mentioning only `orthogonal' when we indifferently mean orthogonal or symplectic. We have tried to write the proofs in a way that formally applies to both cases, switching only the meaning of symbols in a standard way (mostly, the transposes are taken with respect to quadratic forms of different type.)

\subsection{Finite-dimensional representations}
Since Derksen-Weyman's result in \cite{derksen} will be instrumental later on, we carefully review it now. In fact, for us it will be important to understand explicitly some isomorphisms that Derksen-Weyman take as implicit, so we'll actually go through their proof carefully -- but we want to note that the proof is entirely theirs. Let us first recall the definition of symmetric quiver.
\begin{symquiver}
\begin{enumerate}
\item A symmetric quiver $(Q, \sigma)$ is a quiver $Q$ equipped with an involution $\sigma$ on the sets of vertices and arrows such that $\sigma t(\alpha)=h \sigma (\alpha)$, and vice-versa, and that if $\ta=\sigma \ha$, then $\alpha= \sigma (\alpha)$.

\item An orthogonal, resp. symplectic, representation $(V, C)$ is a representation $V$ of $Q$ that comes with a non-degenerate symmetric, resp. anti-symmetric, quadratic form $C$ on its total space $V_\Sigma=\bigoplus_{i \in Q_0} V_i$ which is zero on $V_i \times V_j$ if $j\neq \sigma (i)$, and such that
\begin{equation}\label{alternating}
C(\phi_\alpha v, w)+C(v, \phi_{\sigma (\alpha)} w)=0
\end{equation}
\end{enumerate}
\end{symquiver}

Note that a dimension vector for an orthogonal representation must have $n_i=n_{\sigma (i)}$; we say that such a dimension vector is `compatible.' The theorem is the following:

\begin{derksenweyman}[Derksen-Weyman]
Let $G=\Or (n,\mathbb{C})$ (resp. $\Sp (n, \mathbb{C})$.) Then, to every generalized $G$-quiver $\tilde{Q}$ with dimension vector we can associate a symmetric quiver $Q$ in such a way that the generalized quiver representations of $\tilde{Q}$ correspond bijectively to orthogonal (resp. symplectic) representations of $Q$. Conversely, every symmetric quiver with dimension vector determines a generalized $\Or (n,\mathbb{C})$-generalized quiver.
\end{derksenweyman}

\begin{proof}
Following our convention, we prove only the orthogonal case: the symplectic one is no different. Let $W$ be the standard representation of $G$ (also known as $\mathbb{C}^n$,) and let $C(\cdot,\cdot)$ be the induced symmetric non-degenerate quadratic form on $W$. Then, under the action of $H$, $W$ decomposes into the direct sum
\begin{equation*}
W=\bigoplus W_{\chi_i}
\end{equation*}
where $W_{\chi_i}$ is the isotypic component of the character $\chi_i$ of $H$ (we consider only those characters for which this component is non-empty, so the sum is in fact finite.) The presence of the quadratic form imposes restrictions on this decomposition. In particular, if $v\in W_\chi$ and $w\in W_\mu$, then for all $h\in H$,
\begin{equation*}
C(v,w)=C(h\cdot v, h\cdot w)=\chi(h)\mu(h)C(v,w)
\end{equation*}
Therefore, the restriction of the quadratic form to $W_\chi \times W_\mu$ must be zero if $\chi \mu$ is not trivial. Since the form is non-degenerate, it also follows that for any $\chi$ in the decomposition, $\chi^{-1}$ must also appear, and the restriction to $W_\chi \times W_{\chi^{-1}}$ is non-degenerate. We relabel the character in the decomposition so that $\mu_1, ... , \mu_l$ are the characters with $\mu_i^2=1$, and $\chi_1, ... , \chi_r$ are the maximal number of characters that are not their own inverses or of each other. Also, we denote $V_i=W_{\chi_i}$, and $W_i=W_{\mu_i}$; we've seen that $V_i^*=V_i^t=W_{\chi_1^{-1}}$ where the first equality is the isomorphism we obtain from the quadratic form being non-degenerate. We've just proved that the decomposition of $W$ must be of the form
\begin{equation*}
W=\left( \bigoplus_{i=1}^{i=r} V_i \right) \oplus \left( \bigoplus_{i=1}^{i=r} V_i^* \right) \oplus \left( \bigoplus_{i=1}^{i=l} W_i \right)
\end{equation*}

The centralizer $R$ of $H$ is precisely the group of all orthogonal endomorphisms of $W$ preserving such a decomposition, i.e.,
\begin{equation*}
R=\left( \prod_{i=1}^r R_i \right) \times \left( \prod_{i=1}^l \Or (W_i) \right)
\end{equation*}
To describe $R_i$, note that we've just proved that the quadratic form restricts to a symmetric quadratic form on $V_i\times V_i^*$. Then, $R_i \subset \Or (V_i\times V_i^*)$ is the subgroup respecting the decomposition: its elements are in fact determined by the transformation at $V_i$, since the transformation at $V_i^*$ must be dual to it. Hence, we have an isomorphism
\begin{equation}\label{groupiso}
R=\left( \prod_{i=1}^r \GL (V_i) \right) \times \left( \prod_{i=1}^l \Or (W_i) \right)
\end{equation}

Now, for a vector space $V$, the adjoint representation of $O (V)$ can be identified with $\Lambda^2 (V)$. In the particular case of $V_i\times V_i^*$, the adjoint representation $\Lambda^2 (V_i\times V_i^*)$ of $\Or (V_i\times V_i^*)$, under the action of $R_i$ splits into irreducible summands
\begin{equation*}
\Lambda^2(V_i\times V_i^*)=\Lambda^2 (V_i) \oplus E_i \oplus \Lambda^2 (V_i^*)
\end{equation*}
Here, $\Lambda^2 (V_i)$ is to be seen as the subspace of $\Hom (V_i^*, V_i)\subset \mathfrak{gl}(V_i\times V_i^*)$ which is alternating with respect to the quadratic form $C$, i.e., condition (\ref{alternating}) is satisfied for all $\phi \in \Lambda^2 (V_i)$; analogously $\Lambda^2(V_i^*)\subset \Hom (V_i, V_i^*)$. The summand $E_i\subset \End (V_i)\oplus \End (V_i^*)$ is the subspace satisfying the same alternating condition, which amounts to a pair $(\phi, \psi)$ such that $\psi=-\phi^t$. Again, we have an isomorphism $E_i=\End (V_i)$ induced from the isomorphism for $R_i$ in (\ref{groupiso}) (indeed $E_i=\mathrm{Lie}R_i$, which also shows that this piece is in fact irreducible.) This splitting can easily be checked by writing matrices in two-by-two blocks, and putting $C$ in a standard form.

When we extend the analysis to other summands of $W$, we see the same kind of coupling we found for $E_i$, where the irreducible pieces are subspaces of sums of $\Hom$ spaces. We conclude that under the action of $R$, the adjoint representation of $\Or (n, \mathbb{C})$ splits into factors of the form
\begin{center}
$\Lambda^2(V_i)$, $\Lambda^2(V_i^*)$, $E_i$, $\Lambda^2(W_i)$, \\
$V_{ij}=\Hom (V_i, V_j)$, $V_{i\bar{j}}=\Hom (V_i, V_j^*)$, $V_{\bar{i}j}=\Hom (V_i^*, V_j)$, $W_{ij}=\Hom (W_i,W_j)$\\
$VW_{ij}=\Hom (V_i, W_j)$, $VW_{\bar{i}j}=\Hom (V_i^*, W_j)$
\end{center}
In the second and third line, $i$ and $j$ are to be taken as different. The equalities are isomorphism that we get just as for $E_i$. For example, $V_{ij}\subset \Hom (V_i, V_j)\oplus \Hom (V_i^*, V_j^*)$ is the subspace of $(\phi, \psi)$ such that $\psi=-\phi^t$.

Finally, we construct the quiver. We draw one vertex for each summand in the decomposition of $W$. We will label the ones corresponding to $V_i$ as $q_i$, the ones corresponding to $V_i^*$ as $q_i^*$, and to $W_i$ as $p_i$. For the arrows we must look into the representation $\mathrm{Rep}(\tilde{Q})$ that comes with the generalized quiver. Write $\mathrm{Rep}(\tilde{Q},V)=\bigoplus Z_\alpha$ with $Z_\alpha$ irreducible. Since we've assumed that the trivial representation does not occur in the representation, we may assume that $Z_\alpha$ is not trivial, so that it must be isomorphic to one of the pieces above. To draw the arrows, we go through these pieces one by one as follows
\begin{equation*}
\begin{array}{ll}
\textrm{If $Z_\alpha=\Lambda^2(V_i)$} & \textrm{draw an arrow $g_\alpha=g_\alpha^*: q_i^*\to q_i$} \\
\textrm{If $Z_\alpha=\Lambda^2(V_i^*)$} & \textrm{draw an arrow $g_\alpha=g_\alpha^*:q_i\to q_i^*$}\\
\textrm{If $Z_\alpha=E_i$} & \textrm{draw arrows $g_\alpha=q_i\to q_i$ and $g_\alpha^*:q_i^*\to q_i^*$}\\
\textrm{If $Z_\alpha=\Lambda^2(W_i)$} & \textrm{draw an arrow $g_\alpha=g_\alpha^*:p_i\to p_i$}\\
\textrm{If $Z_\alpha=V_{ij}$} & \textrm{draw arrows $g_\alpha:q_i\to q_j$ and $g_\alpha^*: q_j^*\to q_i^*$}\\
\textrm{If $Z_\alpha=V_{\bar{i}j}$} & \textrm{draw arrows $g_\alpha:q_i^* \to q_j$ and $g_\alpha^*: q_j^*\to q_i$}\\
\textrm{If $Z_\alpha=V_{i\bar{j}}$} & \textrm{draw arrows $g_\alpha:q_i\to q_j^*$ and $g_\alpha^*: q_j\to q_i^*$}\\
\textrm{If $Z_\alpha=W_{ij}$} & \textrm{draw arrows $g_\alpha:p_i\to p_j$ and $g_\alpha^*: p_j\to p_i$}\\
\textrm{If $Z_\alpha=VW_{ij}$} & \textrm{draw arrows $g_\alpha:q_i\to p_j$ and $g_\alpha^*: p_j\to q_i^*$}\\
\textrm{If $Z_\alpha=VW_{\bar{i}j}$} & \textrm{draw arrows $g_\alpha:q_i^*\to p_j$ and $g_\alpha^*: p_j\to q_i$}
\end{array}
\end{equation*}

We now have to define the involution. This is easy with the notation above: switch starred and unstarred elements of the same kind and label: $\sigma (q_i)=q_i^*$, $\sigma (p_i)=p_i$, and $\sigma (g_\alpha)=g_\alpha^*$. It is an easy exercise to verify that with this involution we have a symmetric quiver.

We have left to show that there is a bijection between representations. But this is obvious from our description of the irreducible summands $Z_\alpha$. A representation of the generalized quiver is a vector $v_\alpha \in Z_\alpha$, which is a space of morphisms, and so it is in fact a representation for the arrows. We've seen that $v_\alpha$ is either a morphism in $\Lambda^2 (V_i)$, $\Lambda^2 (V_i)^*$, $\Lambda^2 (W_i)$, in which case the involution we defined fixes the arrow; or it is an element in the other pieces, and it is in fact a pair of morphisms satisfying the condition for an orthogonal representation.

We construct the inverse correspondence. Let $(Q, \sigma)$ be a symmetric quiver. Let $\underline{n}$ be a compatible dimension vector, $V$ be a total space for a representation, and $C$ a quadratic form satisfying the conditions in the definition of orthgonal representation. Using $C$, we can express $V$ as
\begin{equation*}
V=\bigoplus_{i=1}^r (V_i\oplus V_i^*) \oplus \bigoplus_{i=1}^l W_i
\end{equation*}
where the quadratic form restricts to a non-degenerate form on each factor. Then, the group of automorphisms of the orthogonal representation is
\begin{equation*}
\Or (\underline{i})=\prod_{i=1}^r O(V_i\oplus V_i^*) \times \prod_{i=1}^l O(W_i)
\end{equation*}
Here, on $V_i\oplus V_i^*$, the form restricts to the standard pairing of a space with its dual. The orthgonal automorphisms are then the elements of the form $(g,(g^{-1})^t)$, with $g\in \GL_i$. This is precisely the group $R_i$ above, so $\Or (\underline{i})=R$; note further that in $\Or (V)$, $\Or (\underline{i})$ is the centralizer of its center.

The representation space in the generalized quiver is then the space of representations of the symmetric quiver. One can check that it decomposes into summands that are irreducible summads of $\mathfrak{o}(V)$ as an $\mathrm{Ad}\phantom{.} \Or (\underline{i})$-module, and again that those pieces correspond to the pieces we indentified above for the generalized quiver setting.
\end{proof}

An easy corollary of the proof is the following:

\begin{dwequivariance}
The bijection between representations is equivariant with respect to the action of $R=\Or (\underline{i})$.
\end{dwequivariance}

\subsection{Orthogonal $Q$-bundles}
Suppose now that we consider twistings of the morphisms of the quiver. Given a morphism $\phi_\alpha: V_{\ta}\otimes M_\alpha \to V_{\ha}$, the transpose is a map
\begin{equation*}
\phi_\alpha^t: V_{\ha}^*=V_{\sigma (\ha)}\to V_{\ta}^* \otimes M_\alpha^* = V_{\sigma (\ta)}\otimes M_\alpha^*
\end{equation*}
or, equivalently, $\phi_\alpha^t \in \Hom (V_{t\sigma (\alpha)}\otimes M_\alpha, V_{h\sigma(\alpha)})$. For this to be comparable to $\phi_{\sa}$, then, we must necessarily have $M_\alpha=M_{\sa}$.

In light of this, we make the following definition:
\begin{orthbundle}
\begin{enumerate}
\item A twisted symmetric quiver is a symmetric quiver $(Q, \sigma)$ together with a vector bundle $\mathbb{M}_\alpha=\mathbb{M}_{\sa}$ for each orbit of the involution $\sigma$ on the set of arrows $A$.

\item An orthogonal (resp. symplectic) $Q$-bundle is a twisted $Q$-bundle $(\mathbb{V}, \phi)$  where $\mathbb{M}_\alpha=\mathbb{M}_{\sa}$ for all $\alpha$, together with a quadratic form $g\in \mathcal{A}^0 (S^2 \mathbb{V})$ (resp., $g\in \mathcal{A}^0 (\Lambda^2 \mathbb{V})$) that decomposes as
\begin{equation*}
g=\prod_{\begin{subarray}{c}i=\sigma (j)  \\ i<j \end{subarray}}g_{ij}
\end{equation*}
where $g_{ij}\in \mathbb{V}_i^t \otimes \mathbb{V}_j^t$ is non-degenerate, and such that $g(\phi_\alpha (v\otimes m), w)+g(v,\phi_{\sa}(w\otimes m))$ for all $x\in X$, $v\in \mathbb{V}_{\ta,x}$, $w\in \mathbb{V}_{\ha,x}$, and $m \in \mathbb{M}_{\alpha,x}$.
\end{enumerate}
\end{orthbundle}

As before, we will refer to `orthogonal' when we indifferently mean orthogonal or symplectic, unless otherwise stated. Also, we will frequently omit the involution and the twisting bundles when referring to the quiver.

\begin{remarktwisting}\label{remarktwisting}
Let $(\mathbb{V}, \phi)$ be an orthogonal representation of $Q$. Fix, for a moment, an arrow $\alpha$, let $i=\ta$ and $j=\ha$, and assume for simplicity that $j=\sigma (i)$. From the condition on the twisting we have that $\phi_\alpha \in \Hom (\mathbb{V}_i, \mathbb{V}_j) \otimes \mathbb{M}_\alpha$, and $\phi_{\sa}\in \Hom (\mathbb{V}_j,\mathbb{V}_i)\otimes \mathbb{M}_\alpha$. The condition on the form just requires that $g$ restrict to an orthogonal form on $V_{ij}:=\mathbb{V}_i\times \mathbb{V}_j$, which picks a fibrewise orthogonal group $\Or_{ij}:=\Or (V_{ij})$. On this product, $\phi_\alpha$ and $\phi_{\sa}$ determine fibrewise an element $\overline{\phi}_\alpha : \End(V_{ij})\otimes M_\alpha$, and the requirement on the morphisms is essentially that $\overline{\phi}_\alpha \in \mathfrak{o}_{ij}\otimes M_\alpha$. This provides the promised motivation for our earlier definition of twisted generalized quivers.
\end{remarktwisting}

\begin{orthsubspace}
The space of orthogonal $Q$-bundles can be identified as subspace of the plain representations which is invariant under the action of the gauge group determined by $\Or (\underline{i})$. Further, under this identification, the unitary reductions coincide.
\end{orthsubspace}

\begin{proof}
Let $\mathbb{V}$ be a plain representation of the vertices of $Q$. A choice of a symmetric quadratic form $g$ as in the definition corresponds to a reduction of the structure group of $\mathbb{V}$ from $\GL(\underline{i})$ to a choice of orthogonal group $\Or (\underline{i})$ satisfying some additional conditions; it is easy to see that this is the same group as for the finite dimensional case, so we need not repeat it here. Together with the quadratic form on $\mathbb{M}$, this yields a reduction $\pi$ of the structure group of the principal bundle $E^\mathbb{C} \times F^\mathbb{C}$ (for the definitions of $E^\mathbb{C}$ and $F^\mathbb{C}$, refer to section \ref{plain}.)

Having fixed a reduction of the structure group, the involution $\sigma$ and the quadratic form then furnish us with an involution on the space $\Omega^0 (\mathfrak{Rep}(Q,V))$. Namely, given a map $\phi_\alpha$, we've noted above that its transpose can be identified with a map $\phi_\alpha^t: \mathbb{V}_{t\sa}\otimes \mathbb{M}_\alpha \to \mathbb{V}_{h\sa}$. Then, define a map sending $\phi$ to a section $\sigma (\alpha)$, the $\alpha$th component of which equals $-\phi_{\sa}^t$. This is obviously an involution, which we are abusively still denoting by $\sigma$. Explicitly, if $p_\alpha: \Omega^0 (\mathfrak{Rep}(Q,V)) \to \Omega^0 (\mathfrak{Hom}(\mathbb{V}_{\ta}\otimes \mathbb{M}_\alpha, \mathbb{V}_{\ta}))$ is the canononical projection (sending $p_\alpha (\phi)=\phi_\alpha$,) then, $p_\alpha (\sigma(\phi))=-\phi_{\sa}^t$. It is obvious from the definitions that the space $\mathfrak{Rep}^o (Q,V)$ of orthogonal representations of the arrows can be identified with the $-1$ eigenspace of this involution.

Suppose that we have also a hermitian form on $\mathbb{V}$, or, equivalently, of a reduction $E$ of $E^\mathbb{C}$ to a unitary group as in section \ref{plain}. It is easy to see that this can be combined with the reduction to the orthogonal group to yield a reduction to a real orthogonal group. Recall that these reductions can be seen as equivariant maps $E^\mathbb{C} \to \GL (\underline{i})/\U (\underline{i})$, and $E^\mathbb{C} \to \GL (\underline{i})/\Or (\underline{i})$, respectively. Let $E_o^\mathbb{C}$ be the reduced bundle we get from $\pi$, and $E_o^\mathbb{C}\to E^\mathbb{C}$ the inclusion map. This map is $\Or (\underline{i})$-equivariant, so if we compose with the equivariant map we get from the hermitian metric, we get an equivariant map $E_o^\mathbb{C} \to \Or (\underline{i})/\Or (\underline{i}, \mathbb{R})$, where $\Or (\underline{i}, \mathbb{R})=\Or (\underline{i})\cap \U (\underline{i})$. This is the desired reduction to a maximal compact. 
\end{proof}

\subsection{Relation with generalized orthogonal bundles} We shall now prove that the theory of orthogonal $Q$-bundles, and the theory of generalized orthogonal bundles coincide in such a way that the Hitchin-Kobayashi correspondences match. This turns out to be quite easy from our point of view.

\begin{genortheq}\label{genortheq}
Let $\tilde{Q}$ be a $\mathrm{O}(V)$-generalized quiver, and $Q$ be the corresponding symmetric quiver. Then, there is an equivariant bijection between twisted $\tilde{Q}$-bundles and twisted orthogonal bundle representations of $Q$.
\end{genortheq}

\begin{proof}
From Derksen-Weyman's theorem, we have an $\Or (\underline{i})$ equivariant morphism $f: \mathrm{Rep}(\tilde{Q}, M) \to \mathrm{Rep}^o(Q,M,V)$ for some choice of total space $V$. Also, we saw that the structure group of the $Q$-bundles is precisely the group $R$. But then, the product map $\mathrm{id}\times f: (E^\mathbb{C}\times F^\mathbb{C})\times \mathrm{Rep}(\tilde{Q}, M) \to (E^\mathbb{C}\times F^\mathbb{C})\times \mathrm{Rep}(Q, M, V)$ descends to a morphism of the fibered products.
\end{proof}

Set $\mathfrak{Rep}(\tilde{Q},M):=(E\times F)\times_\rho \mathrm{Rep}(\tilde{Q}, M)$ to be the space of morphism representations of the generalized orthogonal quiver. As above, $\mathfrak{Rep}^o(Q,M, V):=(E\times F)\times_\rho \mathrm{Rep}^o(Q, M, V)$ is the space of orthogonal morphism representations of the associated symmetric quiver, and $\mathscr{A}^{1,1}$ is the space of $(1,1)$ connections on $E$.

We endow $\mathscr{A}^{1,1}\times \Omega^0 (\mathfrak{Rep}(\tilde{Q},M))$ with a symplectic form. Since the orthgonal group in question splits as a product, we can directly apply our results of section \ref{hk} to incorporate parameters into this picture. This is simple enough: given the explicit splitting
\begin{equation*}
O(\underline{i})=\left( \prod_{i=1}^r R_i \right) \times \left( \prod_{i=1}^l O(W_i) \right)
\end{equation*}
we choose $r$ positive paremeters $a_1, ..., a_r$, and $l$ positive parameters $b_1, ...,b_l$. Then, these are the parameters we take to weigh the Atiyah-Bott symplectic form $\omega_i$. The moment map on the morphism part is taken as usual.

Now, as a complex subspace of the space of plain representations, $\mathscr{A}^{1,1}\times \Omega^0 (\mathfrak{Rep}^o(Q,M, V))$ is naturally a symplectic space as well. The moment map for the plain case implies choosing a parameter $p_i$ for each vertex.

\begin{symplectomorphism}
Choose parameters such that $p_i+p_{\sigma(1)}=a_i$, if the vertex is not fixed by the involution, and $p_i=b_i$, otherwise. Then, the bijection in lemma \ref{genortheq} is a symplectomorphism between $\mathscr{A}^{1,1}\times \mathfrak{Rep}(\tilde{Q},M)$ and $\mathscr{A}^{1,1}\times \mathfrak{Rep}^o(Q,M, V)$.
\end{symplectomorphism}

\begin{proof}
By construction, the map is obviously a diffeomorphism. Hence, we just have to prove that it preserves the symplectic form. Since the isomorphism of groups identifies the Atiyah-Bott symplectic form, it comes down to checking that the parameter was apropriately chosen. Since the isomorphism only changes the factors corresponding to pairs switched by the involution, we only need to check that parameters match in that case. The dual connection is defined by $A^*=-A^t$; then, on the plain representations, we consider the pair of connections $A\oplus (-A^t)$, and on the generalized quiver, simply $A$. We have
\begin{equation*}
\omega (A\oplus A^*, B\oplus B^*)=p_i\omega (A,B)+p_{\sigma(i)}\omega (-A^t,-B^t)=(p_i+p_{\sigma_i} )\omega (A,B)
\end{equation*}
where, by an abuse of notation, $\omega$ is in each case the apropriate Atiyah-Bott form. Since we want the last one to coincide with $a_i \omega$, which is the form on the generalized quiver, we must have $p_i+p_{\sigma_i}=a_i$.
\end{proof}

\subsection{Stability of orthogonal representations}

We now simplify and characterize the stability condition. A central element $c\in \mathfrak{r}$ can be written
\begin{equation*}
c=\bigoplus_{i=1}^r (\tau_i \mathrm{id}_i\oplus (-\tau_i) \mathrm{id}_{\sigma (i)}) \oplus \bigoplus_{i=1}^l \sigma_i \mathrm{id}_i
\end{equation*}
where $\tau_i \in \mathbb{R}$ and $\sigma_i=\pm 1$. Also, $\chi$ has values in $\mathfrak{r}$, each $\chi_i$ is traceless. Further, if $i$ corresponds to a vertex that is not fixed, $\chi_i=(\psi_i, -\psi_i^t)$. Thus,
\begin{equation*}
\langle c,\chi \rangle = \sum_{i=1}^r (\langle \tau_i,\psi_i\rangle+\langle -\tau_i,-\psi^t_i\rangle)+\sum_{i=1}^l \langle \sigma_i, \chi_i\rangle=\sum_{i=1}^r 2\tau_i \tr \psi_i
\end{equation*}

\begin{genorthst}
A representation $(A, \phi)\in \mathscr{A}^{1,1}\times \Omega^0 (\mathfrak{Rep}(\tilde{Q},M)) $ is stable if for any reduction $\pi$ to a parabolic subgroup $P$, and an anti-dominant character $\chi$ of $P$ such that $\phi (X)\subset \mathcal{F}^-(\chi)$ (cf. section \ref{plain},) we have
\begin{align*}
\sum_{i=1}^r \left(a_i \deg (\pi, \chi_i)-2\tau_i \tr \chi_i \right)+\sum_{i=1}^l b_i \deg (\pi, \chi_i) >0
\end{align*}
\end{genorthst}

We want to simplify this stability condition. In analogy to the general linear case, there is a concrete interpretation of parabolic subgroups of an orthogonal group in terms of special flags. A filtration
\begin{equation*}
0=V^0  \varsubsetneq V^1 \varsubsetneq ... \varsubsetneq V^r=V
\end{equation*}
of a quadratic vector space $V$ is said to be isotropic if $V^{r-k}=(V^k)^{\perp}$ for every $0\leq k \leq r$. The subgroup of $\Or (V)$ fixing such a flag is parabolic, and conversely, every parabolic subgroup is the stabilizer of such a flag.

In the case of a vector bundle $\mathbb{V}$, we easily see that the filtration induced by an antidominant character of a parabolic subgroup is also isotropic. However, this filtration is an isotropic filtration by vector bundles only outside a codimension two submanifold. Since every coherent sheaf defined outside a codimension two submanifold has a unique coherent extension to the whole manifold, what we actually get is an isotropic filtration of $\mathbb{V}$ by subsheaves. 

Let us consider the case of orthogonal representations. Our structure group is
\begin{equation*}
\Or(\underline{i})=\left( \prod_{i=1}^r R_i \right) \times \left( \prod_{i=1}^l \Or(W_i) \right)\
\end{equation*}
where $R_i$ is also an orthogonal group of the product $V_i\oplus V_i^*$. This means that when we look at the splitting $\chi= \oplus \chi_i$, each component is an antidominant character of a parabolic subgroup of an orthogonal group. We have two cases:
\begin{itemize}
\item When $i=\sigma (i)$ is fixed by the involution, i.e., the component $\chi_i$ corresponds to a factor of the form $\Or(W_i)$: then, just as above, $\chi_i$ (and therefore $\chi$) induces an isotropic filtration of $\mathbb{W}_i$.

\item When $i\neq \sigma (i)$ is not fixed by the involution: then, $\chi_i=(\psi_i, -\psi_i^t)$ for some $\psi_i \in \mathfrak{gl}(V_i)$. Note that $\mathbb{V}_i$ and $\mathbb{V}_i^*$ are themselves isotropic subbundles. The isotropic filtration of the direct sum might select subspaces from either of them, e.g., if $\psi_i$ has both a positive and a negative eigenvalues (since the eigenvalues of $-\psi^t$ are the symmetric of the eigenvalues of $\psi$, the first half of the filtration will include an isotropic subbundle containing both subspaces from $\mathbb{V}_i$ and $\mathbb{V}_i^*$.) However, at each step in the filtration, each subbundle can be split into the two vertices. 
\end{itemize}

Using the same method as in section \ref{plain}, we can prove the following:

\begin{morphismsfactororth}
The morphism representation $\phi$ restricts to each subsheaf in the representation.
\end{morphismsfactororth}

This proposition means that each element in the filtration in the definition of stability is actually a sheaf subrepresentation; we call such a subrepresentation an \emph{isotropic} quiver subsheaf. Recall that every orthogonal bundle is isomorphic to its dual, and hence has degree zero. We make the following definition.

\begin{slopestabilityorth}
An orthogonal representation $(E^\mathbb{C}, \phi)$ is \emph{slope stable} if for every isotropic reflexive subsheaf representation $(\mathcal{F},\phi)$ we have
\begin{equation*}
\deg_{a,\tau}^0= \sum_{i=1}^{2r} (a_i \deg (\mathcal{F}_i) - \tau_i \rk \mathcal{F}_i) + \sum_{i=1}^l b_i \deg \mathcal{F}_i <0
\end{equation*}
\end{slopestabilityorth}
Note that the isotropic subsheaf representations are not, by definition, orthogonal representations (since the quadratic form is certainly degenerate.) From now on, when we speak of a subsheaf representation, we implicitly assume it to be `reflexive'.

\begin{equivalencestability}
An orthogonal representation is stable if and only if it is slope stable.
\end{equivalencestability}

\begin{proof}
Let $\mathbb{W}\subset \mathbb{V}$ be a subsheaf of an orthonal bundle. From the short exact sequence
\begin{equation*}
0 \to \mathbb{W}^\perp \to \mathbb{V}^* \to \mathbb{W}^* \to 0
\end{equation*}
we find that $\deg \mathbb{W}^\perp = \deg \mathbb{W}$. Then, given a filtration indued by an antidominant character $\chi$, as in the definition of stability, we have for each $k \leq \lfloor r/2 \rfloor$
\begin{equation*}
(\lambda_{r-k-1} - \lambda_{r-k})\deg (\mathbb{V}^k)^\perp=(\lambda_k-\lambda_{k+1})\deg \mathbb{V}^k
\end{equation*}
Therefore,
\begin{equation*}
\deg (\pi, \chi)=\sum_{k=1}^r (\lambda_k - \lambda_{k+1})\deg \mathbb{V}^k = 2\sum_{k=1}^{\lfloor r/2 \rfloor} (\lambda_k-\lambda_{k+1})\deg \mathbb{V}^k
\end{equation*}
where now each subsheaf in the sum is isotropic (recall that orthogonal bundles have degree zero.) When the vertex is fixed, nothing else needs to be said. When $\mathbb{V}=\mathbb{V}_i\oplus \mathbb{V}_{\sigma (i)}$ is the sum of two exchanged vertices, we have a splitting $\mathbb{V}^k=\mathbb{V}^k_i \oplus \mathbb{V}^k_{\sigma (i)}$, and so,
\begin{equation*}
\deg (\pi, \chi_i)=2 \sum_{k=1}^{\lfloor r \rfloor} (\lambda_k-\lambda_{k+1}) \left( \deg \mathbb{V}_i^k +\deg \mathbb{V}_{\sigma(i)}^k \right)
\end{equation*}
Finally,
\begin{align*}
&\sum_{i=1}^{r'} \left(a_i \deg (\pi, \chi_i)-2\tau_i \tr \chi_i \right)+\sum_{i=1}^l b_i \deg (\pi, \chi_i) \\
&=\sum_{i=1}^{r'} \left(2a_i \sum_{k=1}^{\lfloor r/2 \rfloor} (\lambda_k-\lambda_{k+1}) \left( \deg \mathbb{V}_i^k +\deg \mathbb{V}_{\sigma(i)}^k \right)-2\tau_i \tr \chi_i \right)+\sum_{i=1}^l 2b_i \sum_{k=1}^{\lfloor r/2 \rfloor} (\lambda_k-\lambda_{k+1})\deg \mathbb{W}_i^k\\
&=2 \left( \sum_{k=1}^{\lfloor r/2 \rfloor} (\lambda_k-\lambda_{k+1}) \left( \sum_{k=1}^{2r'} (a_i \deg (\mathbb{V}_i) - \tau_i \rk \mathbb{V}_i) + \sum_{i=1}^l b_i \deg \mathbb{W}_i \right) \right)
\end{align*}
Here we've denoted $r'$ the number to orbits of interchanged vertices, so it's not confused with the number of steps in the filtration. Note that by definition, $\lambda_k<\lambda_{k+1}$, and, again, that the terms involve only isotropic subsheaves. Therefore, if the representation is slope-stable, it is stable.

Conversely, given a isotropic sheaf subrepresentation, we apply the stability condition to the two term flag involving that sheaf.
\end{proof}

Inspired by the analogy with the plain case (see section \ref{plain}), we make the following definition:

\begin{orthss}
Let $(V, \phi)$ be an orthogonal quiver bundle. Then it is \emph{semistable} if for every isotropic subsheaf representation $\mathcal{F}$, we have $\deg_{a,\tau}^0(\mathcal{F})\leq 0$.
\end{orthss}

Given an orthogonal representation, we then have two concepts of (semi) stability for it: as an orthogonal representation of a symmetric quiver $(Q,\sigma)$ (which we defined above), or as a plain representation of the underlying quiver $Q$ (which we defined in section \ref{plain}.)  In fact, these are closely connected. 

\begin{orthplainrelation}\label{orthplainrelation}
Let $(Q, \sigma)$ be an symmetric quiver, and $(\mathbb{V}, \phi)$ be an orthogonal bundle representation. Then,
\begin{enumerate}
\item $(\mathbb{V}, \phi)$ is semistable as plain representation if and only if it is semistable as an orthogonal representation.
\item $(\mathbb{V}, \phi)$ is orthogonally stable if and only if it is an orthogonal sum of mutually non-isomorphic sheaf subrepresentations, each of which is stable as a plain sheaf representation.
\end{enumerate}
\end{orthplainrelation}

For the proof note, that $\mu_{a,\tau} (\mathbb{V})=0$, since the vector bundle is orthogonal, and so it is isomorphic to its dual (a fact we've used before.)

\begin{proof}
\begin{enumerate}
\item One direction is obvious. For the other, suppose $\mathbb{V}$ is orthogonally semistable, let $\mathcal{F}$ be an arbitrary (sheaf) subrepresentation, and denote $\mathcal{E}:=\mathcal{F}\cap \mathcal{F}^\perp$. Then, $\mathcal{E}$ defines an isotropic subrepresentation, and we have the exact sequence
\begin{equation*}
0\to \mathcal{E}\to \mathcal{F}\oplus \mathcal{F}^\perp \to \mathcal{M} \to 0
\end{equation*}
We have the isomorphism $\mathcal{M}=\mathcal{E}^\perp$ (as sheaves), and since $\mathcal{E}$ is isotropic and $\phi$ alternating, $(\mathcal{M}, \phi)$ is a subrepresentation. Indeed, $\phi_\alpha$ restricts to a map $\phi_\alpha^1$ of $\mathcal{E}$, and we can write
\begin{equation*}
\phi_\alpha = \left( 
\begin{array}{cc}
\phi_\alpha^1 & \beta \\
0 & \phi_\alpha^2
\end{array}
 \right)
\end{equation*}
Now, the condition on the maps states that $\phi_{\sigma(\alpha)}=-\phi_\alpha^t$. Since $\phi_{\sigma(\alpha)}$ also restricts to a map in $\mathcal{E}$, this implies that $\beta=0$.

Therefore, $\deg^0_{a,\tau} (\mathcal{F}\oplus \mathcal{F}^\perp )=\deg^0_{a,\tau}  (\mathcal{E})+\deg^0_{a,\tau}  (\mathcal{E}^\perp)$. Now, since the quadratic form $C$ is non-degenerate, it gives an isomorphism $E\simeq E^*$, while we have a short exact sequence $0\to E^\perp \to V\to E^*\to 0$. Hence, $\deg^0_{a,\tau}  (F)=\deg^0_{a,\tau} (E)\leq 0$ by the orthogonal semistability of $V$.

\item Suppose $(\mathbb{V},\phi)$ is orthogonally stable, but not stable as a plain representation, and let $(\mathcal{F}, \phi)$ be a destabilizing subrepresentation.  Using the notation of the previous point, since $\mathbb{V}$ is orthogonally stable, $\mathcal{E}$ is trivial, which means that we have an orthogonal decomposition $\mathbb{V}=\mathcal{F}\oplus \mathcal{F}^\perp$, which is also a decomposition of orthogonal representations, since $\phi$ is alternating and the quadratic form non-degenerate. Actually, each of the representations is orthogonally stable as well, because $\mathbb{V}$ is so (though not necessarily stable as plain representations.) By induction on the rank, we decompose $\mathbb{V}=\perp_i \mathcal{F}_i$ where the $\mathcal{F}_i$ are stable as plain representations. If we have $\mathcal{F}_1\simeq \mathcal{F}_2$, then the embedding $x\mapsto (x,ix)$ gives an isotropic subrepresentation contradicting the stability of $\mathbb{V}$. Conversely, if no two summands are isomorphic, any subrepresentation of maximal degree would be a sum of some of the $\mathcal{F}_i$, and cannot be isotropic (again, the $\mathcal{F}_i$ cannot be isotropic since $\mathbb{V}$ is non-degenerate.)
\end{enumerate}
\end{proof}

The following is an easy corollary of the decomposition in the theorem
\begin{simpleorthstable}
Let $(\mathbb{V}, \phi)$ be an orthogonally stable representation. Then it is stable as a plain representation if and only if it is orthogonally simple (i.e., its only automorphisms are $\pm I$).
\end{simpleorthstable}

This has the following easy consequence:

\begin{onlyfixed}
If an orthogonal representation $\mathbb{V}, \phi)$ is stable as a plain representation, it is trivial on any vertex that is not fixed by the involution, i.e., $\mathbb{V}_i=0$ if $i\neq \sigma (i)$.
\end{onlyfixed}

\subsection{Polystability}

The gauge equations for the orthogonal case are just the projection of the equations for the plain case onto the Lie algebra of the orthogonal group. Thus, if an orthogonal representation already solves the gauge equation for the orthogonal case, it also solves them for the plain case, and so it is polystable as a plain representation. This means, in particular, that we have a splitting
\begin{equation*}
(\mathbb{V}, \phi)=\bigoplus (\mathcal{F}_i, \phi)
\end{equation*}
into stable (plain sheaf) subrepresentations. Now, a given summand ($\mathcal{F}_i, \phi)$ might very well be an orthogonal representation (i.e., the quadratic form might be non-degenerate on its total space), in which case it is orthogonally stable as well, a case which we described above. Otherwise, if the representation is not orthogonal, it must necessarily intersect its orthogonal complement in $(\mathbb{V}, \phi)$, and since stable plain representations are simple, it must be isotropic. Since $(\mathbb{V}, \phi)$ is itself orthogonal, there is a $j\neq i$ such that $(\mathcal{F}_j, \phi)\simeq (\mathcal{F}_i, \phi)^*$, and the quadratic form restricts to the standard orthogonal pairing. In other words, $(\mathcal{F}_i, \phi)\oplus (\mathcal{F}_j, \phi)=(\mathcal{F}_i, \phi)\oplus (\mathcal{F}_i^*, \phi)$ is stable orthogonal representation (but not, of course, stable as a plain representation.) We arrive at the following result:
\begin{psdecomp}\label{psdecomp}
Let $(\mathbb{V}, \phi)$ be an orthogonal representation which solves the Hitchin-Kobayashi correspondence. Then, we have a decomposition
\begin{equation}\label{decomp}
(\mathbb{V}, \phi)= \bigoplus (\mathcal{F}_i, \phi)^{f_i} \oplus \bigoplus ((\mathcal{E}_i, \phi)\oplus  (\mathcal{E}^*_i, \phi))^{e_i} \oplus \bigoplus ((\mathcal{S}_i, \phi)\oplus  (\mathcal{S}^*_i, \phi))^{s_i}
\end{equation}
where $f_i$, $e_i$, and $s_i$ are positive integers,  $(\mathcal{F}_i, \phi)$ are stable orthogonal subrepresentations, $(\mathcal{S}_i, \phi)$ and  $(\mathcal{E}_i, \phi)$ are stable plain representations respectively  isomorphic and not isomorphic to their dual. Further, a given factor is not isomorphic to any other factor in the sum, and the sums $(\mathcal{E}_i, \phi)\oplus (\mathcal{E}^*_i, \phi)$ and $(\mathcal{S}_i, \phi)\oplus  (\mathcal{S}^*_i, \phi)$ endowed with the standard orthogonal pairing.
\end{psdecomp}

\begin{riemanncase}
We can make the following change in the previous composition. Let $(\mathcal{S}_i, \phi)$ be a summand that is isomorphic to its dual. Then, if we choose an $\mathbb{C}$-linear isomorphism $\psi: \mathcal{S}_i^*\simeq \mathcal{S}_i$, it induces an isomorphism $(\mathcal{S}_i, \phi)\oplus  (\mathcal{S}^*_i, \phi) \simeq (\mathcal{S}_i \otimes \mathbb{C}^2, \phi)$ defined by $(f,g)\mapsto f\otimes e_1+\psi (g)\otimes e_2$ (here we're implicitly using that $\mathcal{S}_i$ is locally free outside codimension two.) Recall that such a $\mathbb{C}$ linear isomorphism is equivalent to a pairing on $\mathcal{S}_i$, and in fact we can arrange so that the pairing is skew-symmetric. We will also assume that this isomorphism respects the vertices and the involution on them, in the sense of a representation of a symmetric quiver. Then, the orthogonal pairing on $\mathcal{S}_i\oplus \mathcal{S}^*_i$ coming from $\mathbb{V}$ naturally induces a skew-symmetric pairing on $\mathbb{C}^2$. On the other hand, the maps $\phi_\alpha$ naturally induce maps $\overline{\phi}_\alpha$ on $\mathcal{F}\otimes \mathbb{C}$, determined by the conditions $f\otimes e_i \mapsto \phi_\alpha(f)\otimes e_1$, and $f\otimes e_2 \mapsto \psi \phi_{\sigma(\alpha)} \psi^{-1} (g)\otimes e_2$. This map is now alternating with respect to the symplectic form on $\mathcal{F}$; in other words, the previous decomposition can be written as
\begin{equation*}
(\mathbb{V}, \phi)= \bigoplus (\mathcal{F}_i, \phi)^{f_i} \oplus \bigoplus ((\mathcal{E}_i, \phi)\oplus  (\mathcal{E}^*_i, \phi))^{e_i}  \oplus \bigoplus (\mathcal{S}_i, \phi)^{s_i}
\end{equation*}
where now the $\mathcal{S}_i$ are symplectic representations of the symmetric quiver. In this way, both orthogonal and symplectic representations of a symmetric quiver are necessary, and in this way we're naturally `thrown' into the concept of supermixed quivers, which we'll see in the next section.
\end{riemanncase}

We have in fact fully characterized polystability.

\begin{completesolution}
Let $(\mathbb{V}, \phi)$ be an orthogonal representation. Then, there is a hermitian metric solving the gauge equations if and only if it has a decomposition as in Lemma \ref{psdecomp}.
\end{completesolution}

\begin{proof}
If a representation has a decomposition as in Lemma \ref{psdecomp}, then it is a sum of stable representations, so it solves the plain gauge equations. Since it is already an orthogonal representation, it solves the orthogonal equations.
\end{proof}

We are now only missing one piece in a complete Hitchin-Kobayashi correspondence: to relate polystability in the sense of solving the gauge equations with polystability in the sense of satisfying a linear criterion like that of stability. The missing step is to characterize the Jordan-H\"older filtration associated with a semistable object, as shown in \cite{ggm}. Strictly speaking, our case doesn't fit into that framework, since their correspondence needs to be tweaked in the same sense that the results in \cite{mundet} are tweaked in section \ref{hk}; doing this, however, should be straightforward. More meaningfully, there isn't, at present, any complete correspondence for base manifolds of higher dimension, at least within the point of view we've taken (i.e., the correspondence as a linear symplectic criterion; see, however, \cite{lt}.)

\section{Supermixed Quivers}\label{supermixed}

In this section, we will consider a generalization of symmetric quivers that was originally introduced by Zubkov \cite{zubkov1} \cite{zubkov2}, the \emph{supermixed quivers}. These are flexible enough to allow for a unified treatment of both symplectic and orthgonal symmetries.  However, the case is simple enough that the generalization of our results for symmetric quivers is immediate. More importantly, they introduce another example of a concrete interpretation for generalized quivers.

\subsection{Finite dimensional case}

We will use the definition of Bocklandt \cite{bocklandt}, whose more abstract formalism better suits our treatment. Also, the fact that he uses involutions to pin down the quiver symmetries might very well apply to include other closed linear groups; we hope to come back to this issue in some later paper. Bocklandt's definition is as follows:

\begin{dualizing}
Let $S$ be a finite dimensional complex semisimple algebra, and $M$ an $S$-bimodule. A \emph{dualizing structure on $(S,M)$} consists of two anti-linear involutions $*: S\to S$ and $*: M\to M$ satisfying the compatibility condition:
\begin{equation*}
(amb)^*=b^* m^* a^*
\end{equation*}
for all $a,b\in S$ and $m\in M$. A \emph{dualmod} (DUalized ALgebra and MODule) is a pair $(S,M)$ with a dualizing structure.
\end{dualizing}

The setup above looks less esoteric if we take into account the following fact: given $S$ and $M$ as above (but without the anti-involutions,) then there is a (plain) quiver $Q$ such that $(S,M)\simeq (\mathrm{M} (\underline{i}), \mathrm{Rep} (Q, \underline{i}))$, where $\mathrm{M}(\underline{i})=\bigoplus \End (\mathbb{C}^i)$. (In fact, we have been implicitly using this all through the paper.) It turns out that one can still geometrically interpret dualmods as representations of supermixed quivers.

\begin{supermixed}
\begin{enumerate}
\item A \emph{supermixed quiver} $(Q, \sigma, \epsilon)$ is a symmetric quiver $(Q, \sigma)$ together with a sign map $\epsilon : I \cup A \to \{ \pm 1\}$ such that $\epsilon_i \epsilon_{\sigma (i)}=\epsilon_\alpha \epsilon_{\sa}=1$.
\item A \emph{supermixed representation} is a representation $(V, \phi)$ of $Q$ such that $\phi_\sa = \epsilon_\alpha \phi_\alpha^*$.
\end{enumerate}
\end{supermixed}

As usual, we'll often loosen our tongue, and speak of the supermixed quiver $Q$, omitting the involution and sign map from the notation. The following result is proven in \cite{bocklandt}:

\begin{geometricsupermixed}
To every dualmod $(S,M)$ we can associate a supermixed quiver $Q$ together with a dimension vector $\underline{i}$ such that the supermixed representations of $Q$ are in bijective correspondence with the elements of the dualizing module $D(M)$.
\end{geometricsupermixed}

\begin{pathalgebra}
As is well known, and as the proof above will easily recall, quiver representations can also be characterized as modules of a certain algebra associated with $Q$, called the \emph{path algebra} $\mathbb{C}Q$. In fact, the path algebra is a universal example of the semisimple algebras we have used in proof, and indeed it does \emph{not} come with an associated dimension vector. It is true that we could impose a `universal dualizing structure' on $\mathbb{C}Q$, but this algebra is generally infinite dimensional, and so the dualizing group so determined is, in principle, infinite dimensional as well. It is not presently clear what can be said in general about a generalized quiver associated with such a group. (In fact, these comments apply as well to plain quivers; it is because we want to deal with finite dimensional groups that generalized quivers already determine a dimension vector.)
\end{pathalgebra}

Let $(S,M)$ be a dualmod. Then, the dualizing group of $S$ is the group
\begin{equation*}
D(S)=\{ g\in S \phantom{.}| \phantom{.} g^*g=gg^*=1\}
\end{equation*}
and the dualizing submodule of $M$ is the module of self-dual elements of $M$:
\begin{equation*}
D(M)=\{ m \in M \phantom{.}| \phantom{.} v^*=v \}
\end{equation*}
The main point of these definitions is the following theorem.

\begin{supermixedgeneralized}
The dualizing group $D:=D(S)$ is a complex reductive (in fact, semisimple) Lie group. Let $1=\sum e_i$ be a maximal decomposition by orthogonal idempotents, and $H^0:=\mathbb{C}\{e_i\}$ be the subspace generated by them. Then, the intersection $H=H^0 \cap D$ together with the dualizing module $D(M)$ defines a generalized $D$-quiver setting.
\end{supermixedgeneralized}

It now makes sense to make the following definition.
\begin{generalizedsupermixed}
Let $(S,M)$ be a dualmod, and $D=D(S)$ and $D(M)$ the associated dualizing group and module, respectively. Then, a generalized supermixed quiver is the generalized $D$-quiver determined by the module $D(M)$.
\end{generalizedsupermixed}

\subsection{Supermixed quiver bundles}

The geometric construction of a supermixed quivers makes quite explicit the similarities with the orthogonal representations of symmetric quivers from the previous section. We can define supermixed bundles in the obvious way, using the general definitions of section \ref{generalized}.

The work in the finite dimensional case is already enough for the next correspondence.
\begin{correspondencesupermixed}
Let $\tilde{Q}$ be a generalized supermixed quiver, and $Q$ be the corresponding supermixed quiver. Then, there is an equivariant bijection between twisted $\tilde{Q}$-bundles and supermixed twisted bundle representations of $Q$.
\end{correspondencesupermixed}

The proof goes exectly like the one for symmetric quivers. It is also straightforward to generalize the other results, and we assemble them here for reference.

\begin{supermixedbundle}
Let $\tilde{Q}$ be a generalized supermixed quiver.
\begin{enumerate}
\item The space of bundle representations of $\tilde{Q}$ embeds as a subspace of the representation space of the underlying quiver $Q$. Further, for a particular choice of stability parameters, this embedding is symplectic (i.e., the equivariant bijection in the previous theorem is a symplectomorphism.)
\item A representation is  stable if and only if it is slope semistable, where slope stability is also defined in terms of isotropic subsheaves.
\item A representation is semistable as a supermixed representation if and only if it is semistable as a plain representation.
\item If the representation $(\mathbb{V}, \phi)$ solves the gauge equations, it has a decomposition
\begin{equation*}
(\mathbb{V}, \phi)= \bigoplus (\mathcal{F}, \phi) \oplus \bigoplus ((\mathcal{E}_i, \phi) \oplus (\mathcal{E}^*_i, \phi)) \oplus \bigoplus ((\mathcal{S}_i, \phi) \oplus (\mathcal{S}^*_i, \phi))
\end{equation*}
where $(\mathcal{F}, \phi)$ are orthogonal representations, $(\mathcal{S}_i, \phi)$ a stable plain representations isomorphic to their duals such that $(\mathcal{S}_i, \phi) \oplus (\mathcal{S}^*_i, \phi)$ are symplectic representations, $(\mathcal{E}_i, \phi)$ stable plain representations not isomorphic to their dual, and $(\mathcal{E}_i, \phi) \oplus (\mathcal{E}^*_i, \phi)$ is an orthogonal representation with the standard orthogonal pairing of a space with its dual.
\end{enumerate}
\end{supermixedbundle}

\section{Further Examples}\label{examples}
The goal of the last section was to show that generalized quivers actually yield down-to-earth objects in concrete cases, despite giving an easier setting for moduli problems. We did this by carefully studying the orthogonal and symplectic case. In this section, we quickly mention a few further examples to reinforce the point.

\subsection{Higgs bundles over Riemann surfaces}
A Higgs bundle over a Riemann surface is the case of a single vertex and a single morphism. These provide the simplest examples of the theory, and indeed our observations about twistings apply directly to this case. In all instances, the twisting bundle is $\mathbb{M}=\mathbb{K}$, the canonical bundle of $X$.

Let first $G= \GL (n, \mathbb{C})$, and $H=Z(G)=\left\{ \lambda \mathrm{id} | \lambda \in \mathbb{C}^* \right\}$ be the torus of constant-diagonal matrices. The centralizer of $H$ in $G$ is then all of $G$. Under the action of $G$, the Lie algebra $\mathfrak{g}$ is an irreducible module, and so, necessarily $\mathrm{Rep}(Q,V)=\mathfrak{g}^{\oplus n}$; for the Higgs bundle, we take $n=1$. Since our representation space only has one irreducible component, a representation is especially simple: the vertex symmetry group and the twisting group act separately on the morphism and twisting vector spaces. Then, we must choose a principal $\GL (n,\mathbb{C})$-bundle $E^\mathbb{C}$, and the space of generalized quiver bundles for such a choice of $H$ is
\begin{equation*}
\mathfrak{Rep}(\tilde{Q}, \mathfrak{g})=\mathrm{Ad}(E)\otimes \mathbb{K}
\end{equation*}
Specializing to unitary/hermitian case, we need a $\U (n)$-bundle $E$, and a section $\phi \in (\mathrm{Ad}(E)\otimes \mathbb{K})$. We almost trivially recover the classical case: under the standard representation $\GL (n,\mathbb{C})$ on $\mathbb{C}^n$, $\mathfrak{g}$ identifies with $\End (\mathbb{C}^n)$, and $\mathrm{Ad}(E)$ identifies with $\End (\mathbb{V})$ for some vector bundle $\mathbb{V}$. Then,
\begin{equation*}
\mathfrak{Rep}(\tilde{Q}, \mathfrak{g})=\End (\mathbb{V})\otimes \mathbb{K}=\Hom (\mathbb{V}, \mathbb{V}\otimes \mathbb{K})
\end{equation*}
i.e., a representation of the arrow is just a Higgs field $\phi : \mathbb{V} \to \mathbb{V}\otimes \mathbb{K}$.

The stability for this special case was originally studied by Hitchin in the rank 2 in \cite{hitchin}. As we've mentioned, it is precisely on Riemann surfaces that the stability condition simplifies to a slope condition without extending the theory to coherent sheaves. In fact, defining $\mu =\deg (\mathbb{V})/\rk (\mathbb{V})$ for any vector bundle $\mathbb{V}$, Hitchin found that the correct semistability condition is that $\mu (\mathbb{E}) \leq \mu (\mathbb{V})$ for every proper subbundle $0\neq \mathbb{E} \subset \mathbb{V}$ that is $\phi$ invariant in the sense that $\phi (\mathbb{E})\subset \mathbb{E}\otimes \mathbb{K}$, stability corresponding to a strict inequality. General polystability can then be described as a splitting into a sum of non-isomorphic stable bundles of the same slope.

The case for closed linear groups $G\subset \GL (n, \mathbb{C})$ follows readily in an analogous manner. We take $H=Z(G)=Z(\GL(n, \mathbb{C})) \cap G$, finding that the centralizer of $H$ will again be the whole of $G$, and obviously $\mathfrak{g}$ is irreducible as an $\mathrm{Ad}G$-module. We find that under the standard representation, the morphism representation is a section $\phi \in (E\times_\mathrm{Ad} \mathfrak{g})\otimes \mathbb{K}$. These have been amply studied.

We note that when the base manifold is of higher dimension (i.e., $X$ is not a curve,) Simpson established the stability conditions for Higgs bundles. However, on higher dimensional base manifolds one imposes an integrability condition on the Higgs field. Such integrability condition is required for the non-abelian Hodge theorem to hold, but it is unnatural from the point of view of quiver bundles, which we're taking here.

\subsection{G-Higgs bundles}

Let $G$ be a real reductive Lie group, and $K$ a maximal compact, with Lie algebras $\mathfrak{g}$ and $\mathfrak{k}$, respectively. \emph{$G$-Higgs bundles} are generalizations of the previous example that have been the object of intense study which encompass the real case; we hope to explore the case of real Lie groups in a later stage.    Note that $\mathfrak{g}$ has a Cartan decomposition $\mathfrak{g}=\mathfrak{k}\oplus \mathfrak{m}$, where $\mathfrak{m}$ is an $\mathrm{Ad}\phantom{.}K$-module (this is the isotropy representation.) A $G$-Higgs bundle is a $K^\mathbb{C}$-principal bundle $E$ together with an element $\phi \in E(\mathfrak{m}^\mathbb{C}$.) Suppose now that $G$ has a complexification. Then, $\mathfrak{m}^\mathbb{C}$ sits naturally inside $\mathfrak{g}^\mathbb{C}$ as an $\mathrm{Ad}\phantom{.}K^\mathbb{C}$-module. In this setting, this clearly defines a generalized $G$-quiver.

Note that in the definition of generalized quiver, we could prefectly well have take $G$ to be a real reductive group, and thus included a general $G$-Higgs bundle (i.e., one for which the complexification of $G$ doesn't necessarily exist.) In fact, in that situation, both $\mathfrak{g}$ and $G$ come with a (global) Cartan decomposition $\mathfrak{g}=\mathfrak{k}\oplus \mathfrak{m}$, and $G=K\exp (\mathfrak{m})$. Note that the first decomposition is \emph{not} a decomposition of Lie algebras, but a decomposition of $\mathrm{ad} \mathfrak{h}$- (or $\mathrm{Ad}H$-) modules. Now, $K$ is the centralizer in $G$ of its center, i.e., $Z_G (Z(K))=K$.) Also, $\mathfrak{m}^\mathbb{C}$, as a $\mathrm{Ad} K$-module decomposes into two pieces isomorphic to $\mathfrak{m}\subset \mathfrak{g}$. Therefore, it makes sense to define the following generalized $G$-quiver bundle: $(Z(K), K, \mathfrak{m}^\mathbb{C})$. 

\subsection{$\mathfrak{Q}$-mixed quivers}

Generalizations of symmetric and supermixed quivers have already been studied in the finite dimensional case, from an algebraic point of view. In \cite{mixed}, Lopatin and Zubkov introduce a unifying notion of $\mathfrak{Q}$-mixed quivers. Essentially, these quivers include more general symmetries by explicitly constructing geometric instances of generalized quivers for very particular choices of the reductive abelian group $H$.

We first fix some notation. We'll denote by $J$ the standard sympletic form in $\mathbb{C}^{2n}$, that is,
\begin{equation*}
J=\left(
\begin{array}{cc}
0 & I_n \\
-I_n & 0
\end{array}
\right)
\end{equation*}
Keeping this in mind, we define:
\begin{itemize}
\item $S^+(n):=\{A\in \GL(n)|A^t=A\}$
\item $S^-(n):=\{A\in \GL(n)|A^t=-A\}$
\item $L^+(n):=\{A\in \GL(n)|AJ\in S^+\}$
\item $L^-(n):=\{A\in \GL(n)|AJ\in S^-\}$
\end{itemize}
These are the Lie algebras of the orthogonal and symplectic groups and their complements.

\begin{mixed}
A \emph{mixed quiver setting} is a quintuple $\mathfrak{Q}=(Q,\mathbf{n},\mathbf{g},\mathbf{h},\sigma)$ where $Q$ is a quiver, $\mathbf{n}$ is a dimension vector for $Q$, $\mathbf{g}=(g_i)$ is a symbol sequence indexed by the vertices of $Q$ with $g_i\in \{ \GL,\Or, \SO, \Sp, \SL \}$, $\mathbf{h}=(h_\alpha)$ are symbols indexed by the arrows of $Q$ with $h_\alpha \in \{ M, S^+, S^-, L^+, L^- \}$, and $\sigma$ is an involution on the sets of vertices and arrows. These are subject to the conditions:
\begin{enumerate}
\item if $g_i=\Sp$, then $n_i$ is even;
\item if $h_\alpha \neq M$, then $n_{\ta}=n_{\ha}$;
\item if $\alpha$ is a loop and $h_\alpha=S^+$ or $S^-$, then $g_{\ta}=\Or$ or $\SO$;
\item if $\alpha$ is a loop and $h_\alpha=L^+$ or $L^-$, then $g_{\ta}=\Sp$.
\item $n_{\sigma(i)}=n_i$;
\item if $g_i=\Or, \SO, \Sp$, then $\sigma{i}=i$;
\item if $\alpha$ is not a loop and $h_\alpha\neq M$, then $\sigma (\ta)=\ha$ and $h_\alpha=S^+$ or $S^-$.
\end{enumerate}
\end{mixed}

The following definition, though not quite the definition in \cite{mixed}, is its geometric interpretation.

\begin{qmixedrep}
Let $\mathfrak{Q}$ be a mixed quiver setting. A \emph{$\mathfrak{Q}$-mixed representation} is a representation $(V,\phi)$ of $Q$, where $V_i=\mathbb{C}^{n_i}$, with the additional data:
\begin{enumerate}
\item $V_{\sigma{i}}=V_i^*$;
\item if $g_i=\Or$ or $\SO$, then $V_i$ comes with the standard orthogonal form, and if $g_i=\Sp$, the $V_i$ comes with the standard symplectic form;
\item if $i\neq \sigma (i)$, Then $V_i\oplus V_{\sigma(i)}$ comes with the standard orthogonal pairing;
\item if $g_i=\SL$ or $\SO$, then $V_i$ comes with a volume form;
\item if $h_\alpha =M$, $S^+$, $S^-$, $L^+$, or $L^-$, then with respect to the previous conditions, $\phi_\alpha \in M(n_\alpha)$, $S^+(n_\alpha)$, $S^-(n_\alpha)$, $L^+(n_\alpha)$, or $L^-(n_\alpha)$, respectively.
\end{enumerate}
\end{qmixedrep}

The reduction to the generalized quiver setting is straightforward if we note the following: a $\mathfrak{Q}$-mixed representation is identified by an element of $H(\mathbf{n},\mathbf{h})=\bigoplus H_\alpha$, where $H_\alpha =M(n_\alpha)$, $S^+(n_\alpha)$, $S^-(n_\alpha)$, $L^+(n_\alpha)$, or $L^-(n_\alpha)$ according to whether $h_\alpha =M$, $S^+$, $S^-$, $L^+$, or $L^-$, respectively; and the symmetry group for such a representation is $G(\mathbf{n},\mathbf{g})=\prod G_i$, where $G_i=\GL(n_i)$, $\SL(n_i)$, $\Or(n_i)$, $\SO(n_i)$, or $\Sp(n_i)$ according to whether $g_i=\GL$, $\SL$, $\Or$, $\SO$, or $\Sp$, respectively.

Note that these representations include both plain representations, as well as orthgonal and symplectic representations of symmetric quivers. Signed quivers correspond to the case where $g_i=\GL, \Or, \Sp$ for all $i$; mixed quivers correspond to $g_i=\GL$ and $h_\alpha=M$ for all vertices $i$ and arrows $\alpha$. Note also that the full generality of the inclusion of $\SO(n)$ and $\SL(n)$ is not used, from the point of view of generalized quivers, in the sense that we do not allow for a general choice of one of their abelian subgroups $H$ (see the next example.) This is due mostly to the fact that Lopatin-Zubkov are interested in their role in yielding \emph{semi-invariants} of representations of $\Or (n)$ and $\GL (n)$, respectively.

\subsection{`Traceless quivers'}

Symmetric quivers are atypical in the sense that they give a full characterization of orthogonal generalized quivers. It is much harder to characterize generalized quivers for a more general group (that is, a result describing geometrically the generalized quiver for any choice of reductive abelian subgroup.) As we've mentioned, this full generality is not even used for the case $\SL(n)$, and we want to use these to exemplify the complexity of the situation.

Let $H$ be an abelian reductive subgroup of $\SL(V)$. As in the orthogonal case, we can decompose $V$ in isotypic components $V=\bigoplus V_\mu$ where $\mu$ is a character of $H$. Noting that we can find a basis of $V$ with respect to this splitting, and denoting by $\omega$ the volume form of $V$ determined by the choice of a special linear group, we have
\begin{equation}
\omega (g e_1,...,ge_n)=\prod_i \mu_i^{n_i} \omega (e_1,...,e_n)
\end{equation}
where the $\mu_i$ are the distinct characters in the representation, and $n_i$ the respective multiplicity. These characters must then satisfy $\prod \mu_o^{n_i}=1$. This condition is much weaker than the condition for the orthogonal group, and accounts for the difficulty of completely characterizing this case.

One possibility is that for every character $\mu_i$, we have $\mu_i^{n_i}=1$. In that case, it is easy to geometrically interpret the generalized quiver: draw one vertex for every character in the representation; let $\mathbf{n}=(n_i)$ be the dimension vector (where $n_i$ are the multiplicities.) The Lie algebra $\mathfrak{sl}(V)$ splits into summands $\Hom (V_i,V_j)$ if $i\neq j$, and $\End_0(V_i)$ (traceless endomorphisms. Thus, the representation space in the definition of generalized quiver can be interpreted in terms of arrows between the vertices just drawn.

It is easy to see that the previous situation must not hold. Inside $\SL(2)$, take $H$ to be the group of all matrices of the form
\begin{equation*}
\left(
\begin{array}{cc}
\lambda & 0\\
0 & \lambda^{-1}
\end{array}
\right)
\end{equation*}
Then, $H$ is its own centralizer, so that $R=H$, and under the action of $H$, $\mathfrak{sl}_n$ splits as above. We can therefore still interpret the generalized quiver geometrically as a quiver with two vertices, but the group of symmetries now pairs the two vertices together (the bundles are dual line bundles.)

\section{The Hitchin-Kobayashi correspondence}\label{hk}
Our version is essentially of the same kind as various general Hitchin-Kobayashi correspondences in the literature: \cite{bradlow}, \cite{banfield}, \cite{mundet}, \cite{bgm}, \cite{lt} (with increasing generality.) Unfortunately, none of these cover our case because of the presence of the parameters in the moment map for the connections part. There's also \cite{ag}, which covers precisely the case of quiver bundles, but only for a very particular choice of gauge group. Still, with the exception of \cite{lt}, the proof of our correspondence is essentially the same as all these articles. In fact, our proof will be very cursory, mostly for the purposes of outline, since it is a straightforward adaptation of previous proofs.

\subsection{The correspondence}

Let $X$ be a compact K\"ahler manifold, $K$ a compact Lie group splitting as finite product $K=\prod K_i$, and $E$ a holomorphic principal $K$ bundle. Given a K\"ahler manifold $F$ with a Hamiltonian left $K$ action $\sigma$, we can form the fibration $E(F)=E\times_\sigma F$.

\subsubsection{The gauge equations} We are interested in the space $\mathscr{A}^{1,1}$ of $K$-connections on $E$, and the space $\mathscr{S}$ of holomorphic global sections of $E(F)$, both properly endowed with symplectic structures.

On the space of sections, we can induce a moment map by fibrewise extension of the symplectic form on $F$, i.e.,
\begin{equation*}
\omega (\phi_1, \phi_2)=\int_X \omega_F (\phi_1(x), \phi_2(x))
\end{equation*}
where $\omega_F$ is the symplectic form on $F$.

For the form on the space of connections, we want to be more careful. Since $K$ splits as a product, $E$ splits accordingly as $E=\prod E_i$ where $E_i$ is a $K_i$ bundle, and $K$-connection splits as a sum $A=\bigoplus A_i$ where $A_i$ is a $K_i$ connection on $E_i$. If for each $E_i$ we take the Atiyah-Bott symplectic form $\omega_i$, and $a_i$ is a collection of positive numbers, then $\omega_E=\sum a_i \omega_i$ is a symplectic form on $\mathscr{A}^{1,1}$.

Now, the gauge group of $E$ acts naturally on $\mathscr{A}^{1,1}$ in a symplectic fashion. Assume that a moment map for the action of $K$ on $F$ exists. It is easy to see that it extends fibrewise to a moment map $\mu$ on $\mathscr{S}$. On the other hand, on the connection part, for each $i$, the Atiyah-Bott moment map is defined, and the moment map for our symplectic form is then the weighted sum of all of these, i.e., its $i$th component is $\left( \mu_E \right)_i =a_i \Lambda F_i$, where $F_i$ is the curvature of the connection $A_i$, and $\Lambda : \mathcal{A}^* \to \mathcal{A}^{*-2}$ is the adjoint of the map that wedges by $\omega_i$.

The gauge equations just define level sets of the moment map on the product $\mathscr{A}^{1,1}\times \mathscr{S}$. Given a collection of central elements $c_i\in \mathfrak{k}_i$, where $\mathfrak{k}_i$ is the Lie algebra of $K_i$, then the gauge equations are
\begin{equation*}
a_i\Lambda F_i +\mu_i (\phi)=c_i
\end{equation*}
In this equation, as always, we're using a hidden parameter, the implicit choice of an equivariant  isomorphism $\mathfrak{k}\simeq \mathfrak {k}^*$ of the Lie algebra with it dual.

\subsubsection{Stability}
Suppose a faithful representation $K\to U(V)$ is given. Given a pair $(\sigma, \chi)$ of a holomorphic reduction $\pi: X\to E(G/P)$ to a parabolic subgroup $P$, and and anti-dominant character $\chi$ of $P$, there is a codimension-two submanifold over which the character $\chi$ induces a holomorphic filtration of the associated fibre bundle $\mathbb{V}=E(V)$:
\begin{equation*}
0\subset \mathbb{V}^1 \subset ... \subset \mathbb{V}^j \subset ... \subset \mathbb{V}^r=\mathbb{V}
\end{equation*}
where $\lambda_1< ... <\lambda_j<...<\lambda_r$ are the eigenvalues of $\rho(\chi)$, and $\mathbb{V}^{\lambda_j}=\bigoplus_{i\leq j} \mathbb{V}(\lambda_i)$ is the sum of all eigenbundles with $\lambda_i\leq \lambda_j$. Given such a pair, we define
\begin{equation*}
\deg (\pi, \chi)=\lambda_r \deg (\mathbb{V}^r)+ \sum_{k=1}^{r-1} (\lambda_k - \lambda_{k+1})\deg (\mathbb{V}^k)
\end{equation*}
where $\deg (\mathbb{V})$ is the degree of the vector bundle.

Now, since our structure group is actually a product of groups, and our connection splits accordingly, nothing demands that we take the Atiyah-Bott form for the total bundle directly. In fact, just above we took a weighted sum of the Atiyah-Bott froms on the vertices. A reduction $\pi$ induces a reduction $\pi_i$ on each of the vertices, and an anti-dominant character obviously splits $\chi=\oplus \chi_i$, like the Lie algebra. But then, the maximal weight changes accordingly, and in fact we should consider instead
\begin{equation*}
\deg_a (\pi, \chi)= \sum a_i \deg (\pi_i, \chi_i)
\end{equation*}
Here, the positive numbers $a_i$ are the parameters for the moment map, as above. We'll call this the `$a$-degree.'

Consider now the action of $K$ on $F$, and for any $x\in F$ and $k\in \mathfrak{k}$ let
\begin{equation*}
\lambda_t (x, k)=\langle \mu (\exp (itk)x), k\rangle
\end{equation*}
where $\langle \cdot, \cdot \rangle$ is the canonical pairing of $\mathfrak{k}$ with its dual. Then, the \emph{maximal weight} of the action of $k$ on $x$ is
\begin{equation*}
\lambda (x,k)=\lim_{t\to \infty} \lambda_t (x,k)
\end{equation*}
This number plays the role of the Hilbert-Mumford criterion in the K\"ahler setting.

Finally, given a section $\phi \in \Omega^0 (E(F))$ of the associated bundle with fibre $F$, and a central element $c\in \mathfrak{k}$, the total $c$-degree $(\sigma, \chi)$ is defined as
\begin{equation*}
T_\phi^{c_i}(\pi, \chi)=\deg_{a} (\pi, \chi) + \int_X \lambda (\phi(x), -ig_{\pi, \chi}(x))+\langle i\chi, c\rangle \mathrm{Vol} (X)
\end{equation*}
Here, $g_{\pi, \chi}\in \Omega^0 (E\times_\mathrm{Ad}i\mathfrak{k})$ is the fibrewise dual of $\chi$. The total degree is allowed to be $\infty$. Henceforward, as before, we'll always assume that the volume of $X$ is normalized to one.

Finally, we define stablity:
\begin{generalstability}
Let $c_i\in \mathfrak{k}_i$ be central elements. A pair $(A,\phi)\in \mathscr{A}^{1,1}\times \mathscr{S}$ is $c_i$-stable if for any submanifold $X_0\subset X$ of complex codimension 2, for any parabolic subgroup $P$ of $G$, for any holomorphic reduction $\pi \in \Gamma (X_0, E(G/P))$ defined on $X_0$, and for any antidominant character $\chi$ of $P$ we have
\begin{equation*}
T^{c_i}_\phi (\pi, \chi)>0
\end{equation*}
\end{generalstability}

For general observations on the notion of stability, see the end of section \ref{plain}.

\subsubsection{The statement}
We need a technical, but important definition:

\begin{generalsimplicity}
A pair $(A,\phi)\in \mathscr{A}^{1,1}\times \mathscr{S}$ is infinitesimally simple if no semisimple element in $\mathrm{Lie}(\mathscr{G}_G)$ stabilizes $(A,\phi)$.
\end{generalsimplicity}

The theorem is as follows:

\begin{hkcorrespondence}
Let $(A,\phi)\in \mathscr{A}^{1,1}\times \mathscr{S}$ be an infinitesimally simple pair. Then $(A,\phi)$ is stable if and only if there is a gauge transformation $g \in \mathscr{G}_G$ such that $(B,\psi)=g\cdot (A, \phi)$ solves the gauge equations
\begin{equation}
a_i\Lambda F_i +\mu_i (\phi)=c_i
\end{equation}
Furthermore, if two different $g, g'\in \mathscr{G}_G$ yield a solution, then there exists a $k\in \mathscr{G}_K$ such that $g'=kg$.
\end{hkcorrespondence}

The proof of this correspondence takes up the rest of this section. The general strategy is standard, and in terms of symplectic geometry can be described as the sequence of steps: stability $\Rightarrow$ properness of the integral $\Rightarrow$ zero of the moment map $\Rightarrow$ stability.

\subsection{Preliminaries}

\subsubsection{The integral of the moment map}

The central construction in the proof is that of the integral of the moment map; this is a rather general construction, and it is the (infinite-dimensional) K\"ahler analogue of Kempf-Ness map for smooth projective varieties. For proofs we refer the reader to \cite{mundet}.

Let $H$ be a Lie group, and suppose there is a complexification $G$ for which the inclusion $H\hookrightarrow G$ induces a surjection $\pi_1(H)\twoheadrightarrow \pi_1(G)$. Note that we are not assuming finite dimensionality, since in our case, the groups are the inifinite dimensional gauge groups.) Let $M$ be a K\"ahler manifold on which $H$ acts respecting the structure, and for which a moment map $\mu: M\to \mathfrak{h}^*$ exists. For a fixed point $p\in M$, we define a 1-form $\sigma^p$ on $L$ by the formula
\begin{equation*}
\sigma^p_g(v)=\langle \mu(g\cdot p), -i\pi(v)\rangle
\end{equation*}
where $g\in G$, $v\in T_gG$, and $\pi:\mathfrak{h}\oplus i\mathfrak{h}\to i\mathfrak{h}$ is the projection onto the second factor. Then, $\sigma$ is exact (it is here that the surjection $\pi_1(H)\twoheadrightarrow \pi_1(G)$ is needed,) and we denote by $\Psi_p:G\to \mathbb{R}$ the unique function such that $d\Psi_p=\sigma^p$, and $\Psi_p(1)=1$. It turns out that the $\Psi_p$ fit together into a smooth function $\Psi:M\times G \to \mathbb{R}$, which we call the integral of the moment map.

The properties of this map are described in the following proposition.

\begin{integralproperties}
Let $p\in M$ be any point, and $s\in \mathfrak{h}$.
\begin{enumerate}
\item $\Psi(p, \exp(is))=\int_0^1 \langle \mu(g\cdot p), s\rangle dt=\int_0^1\lambda_t(p,s)dt$.
\item $\partial_t \Psi (p,\exp(its))|_{t=0}=\langle \mu(p),s\rangle=\lambda_0 $.
\item $\partial^2_t\Psi (p,\exp (its))|_{t=t_0}\geq 0$ for any $t_0 \in \mathbb{R}$, with equality if and only if $\mathcal{X}_s(\exp (it_os)\cdot p)=0$, where $\mathcal{X}_s$ is the vector field generated by $s$.
\item $\Psi (p, \exp(its)\cdot p)\geq (t-t_0)\lambda_t(p,s)+C_s(p,t_0)$ for any $t_0 \in \mathbb{R}$, where $C_s$ is a continuous function in all variables.
\item $\Psi (p,g)+\Psi(g\cdot p, h)=\Psi (p, hg)$ for any $g,h \in G$.
\item $\Psi (h\cdot p, g)=\Psi (p, h^{-1}gh$ and $\Psi (p, hg)=\Psi(p,g)$ for any $h\in H$, and $g\in G$.
\item $\Psi (x,1)=0$.

\end{enumerate}
\end{integralproperties}

Together with the convexity proven in the previous proposition, the next lemma is the fundamental property in the proof:

\begin{criticalintegral}
An element $g\in G$ is a critical poit of $\Psi_p$ if and only if $\mu (g\cdot p)=0$.
\end{criticalintegral}

\subsubsection{Equivalence of $C^0$ and $L^1$ norms}
As usual, we'll need to complete spaces of smooth maps by Sobolev norms, to get spaces that are flexible enough. In general, we want twice-differentiability, and the $L^p$ norm needs to satisfy a bound coming from the Sobolev multiplication theorem. In particular, if $n=\dim  X$, we must choose $p>2n$. The proof of the correspondence involves a properness argument on the integral of the moment map, and to prove such properness we have to fiddle with norms. In particular, we'll require an equivalence between $C^0$ and $L^1$ estimates. Choose $B<0$; we will need to restrict to the subset
\begin{equation*}
\mathscr{M}_{2,B}^p=\left\{ s \in L^p_2(E\times_{\mathrm{Ad}} \mathfrak{k}) |\textrm{\phantom{m}} || \mu^c (\exp (s)(A,\phi))||_{L^p}^p\leq B \right\}
\end{equation*}

\begin{equivalencenorms}\label{equivalencenorms}
There are two constants $C_1,C_2>0$ (which implicitly depend on $B$ and on the parameters $a_i$) such that for all $s\in \mathscr{M}_{2,B}^p$ one has $\sup |s| \leq C_1 ||s||_{L^1}+C_2$.
\end{equivalencenorms}

For a proof this lemma, check \cite{ag} section 3.5, which does not use anything specific to the general linear group.

\begin{mainest}
The integral of the moment map $\Psi^c$ satisfies the $C^0$ \emph{main estimate} if there are constants $C_1,C_2> 0$ such that
\begin{equation*}
\sup |s| \leq C_1 \Psi^c (\exp (s))+C_2
\end{equation*}
If the same condition is verified with $\sup |s|$ replaced with the $L^1$ norm, then we say $\Psi^c$ satisfies the \emph{$L^1$ main estimate}.
\end{mainest}

The following is an easy corollary of Lemma \ref{equivalencenorms}
\begin{equivalenceestimates}
In $\mathscr{M}_{2,B}^p$, the integral of the moment map satisfies the $C^0$ main estimate if and only if it satisfies the $L^1$ main estimate.
\end{equivalenceestimates}

As we mentioned, the point here is that the proof of the correspondence demands the properness of the integral of the moment map in the weak topology of the infinite dimensional Lie algebra involved. The main estimate is only a requirement that implies properness, but it more easily serves as an intermediary step in the proof.

\begin{properintegral}\label{properintegral}
If $\Psi^c$ satisfies the main estimate, then $\Psi^c$ is proper in the weak topology of $L^p_2 (E\times_\mathrm{Ad} \mathfrak{k})$.
\end{properintegral}

\begin{proof}
This lemma is proven by contradiction, and is precisely the same as \cite{bradlow} section 3.14, or \cite{ag} section 3.30.
\end{proof}

\subsubsection{Minima in $\mathscr{M}_{2,B}^p$} Our restriction to  $\mathscr{M}_{2,B}^p$ only makes sense if we can prove that the minima in this subset are in fact minima in the whole of $\mathscr{M}_2^p$. Suppose $(A,\phi)$ is a simple pair, and that $s$ minimizes the integral in $\mathscr{M}_2^p$. Define the operator $L: L_2^p(E\times_\mathrm{Ad}i\mathfrak{k})\to L^p(E\times_\mathrm{Ad}i\mathfrak{k})$ as
\begin{equation*}
L(u)=i\frac{d}{dt}\left. \mu^c (\exp (tu)(B,\theta)) \right|_{t=0}=i\langle d\mu^c, u\rangle (B,\theta) = i \sum a_i \langle d\mu_i , u_i \rangle +i\langle d\mu_S, u \rangle
\end{equation*}
Each $\langle d\mu_i , u_i \rangle$ is a Fredholm operator with index zero (indeed, up to a compact operator, it is $\partial_B^*\partial_B$, cf. \cite{bradlow}.) But, up to a compact operator, $L$ is a linear combination of these, so it is itself a Fredholm operator of index zero. We prove that it is also injective, implying that it is surjective. In fact, if $L(u)=0$,
\begin{equation*}
0=\langle iL(u), -iu\rangle=||\mathscr{X}_{-iu}(B,\theta)||^2
\end{equation*}
which implies that $-iu$ leaves $(B,\theta)$ fixed, and simplicity of $(A,\phi)$ now implies that $u=0$.

Knowing that $L$ is surjective, we conclude that there must be an $u$ such that $L(u)=-i\mu^c (B, \theta)$. A standard argument originally due to Simpson then shows that $\mu^c (B,\theta)=0$, cf. \cite{bradlow} or \cite{ag}.

\subsection{Stability implies main estimate}
We start with a lemma.

\begin{contradictingsequence}
If the integral of the moment map does not satisfy the main estimate, then there is an element $u_\infty \in L^2_p (E\times_\mathrm{Ad} \mathfrak{k})$ such that $\lambda ((A,\phi), -iu_\infty)\leq 0$.
\end{contradictingsequence}

\begin{proof}
Let $C_j$ be a sequence of positive constants diverging to infinity. We start by finding a sequence $(s_j)$ in $L^2_p (E\times_\mathrm{Ad} \mathfrak{k})$ such that $||s_j||_{L^1}\to \infty$ and $||s_j||_{L^1} \geq C_j \Psi (\exp{s_j})$ (cf. \cite{ag} Lemma 3.43.) With such a sequence in hand, we set $l_j=||s_j||_{L^1}$, and $u_j=s_j/l_j$, so that $||u_j||_{L_1}=1$ and $\sup |u_j|\leq C$. We can assume that $\lim_j \lambda^t_{a_i,\tau_i}((A,\phi),-iu_j)$ exists.

Using the convexity of the integral of the moment map, and the fact that $X$ is compact, we have
\begin{equation}\label{est}
\frac{l_j-t}{l_j}\lambda^t_{a_i,\tau_i}(A,-iu_j)+\frac{1}{l_j}\int_0^t \lambda^l_{a_i,\tau_i} (A,-iu_j)dl \leq C
\end{equation}
for some constant $C$.

For a principal bundle $E$ and a connection $A$ on $E$, we have (cf. \cite{mundet}):
\begin{equation*}
\lambda_t(A,s)=\int_X\langle\Lambda F_A, s\rangle +\int_0^t||\exp (ils)\bar{\partial}(s)\exp(ils)||dl
\end{equation*}
where $\lambda_t$ is the finite-time maximal weight for the Atiyah-Bott moment map. It easily follows, then, that in our case we have
\begin{equation*}
\lambda_{a_i,\tau_i}^t=\sum a_i \left(\int_X\langle\Lambda F_{A_i}, s\rangle +\int_0^t||\exp (ils)\bar{\partial}(s_i)\exp(ils)||dl  \right)
\end{equation*}

Using this and (\ref{est}) (recall that the curvature is bounded,) we can prove that $\sum a_i ||\bar{\partial}((u_j)_i
)||_{L^2}$ is bounded, and so $u_j\in L^2_1$. After passing to a subsequence, $u_j\to u_\infty$ weakly in $L^2_1$, since the $u_j$ belong o the unit ball. As the embedding $L^2_1\hookrightarrow L^2$ is compact, the convergence is also strong in $L^2$, and $u_\infty \neq0$ because of the uniform bound on the $C^0$ norm of the $u_j$. To see that $\lambda_{a_i,\tau_i}^t((A,\phi),-iu\infty)\leq 0$, see \cite{ag}.
\end{proof}

Using methods due to Uhlenbeck-Yau \cite{uy} (cf.\cite{bradlow} and \cite{popovici},) we can prove the the element $u_\infty$ in the lemma has almost everywhere constant eigenvalues, and that it defines a filtration of $V$ by holomorphic subbundles in the complement of a complex codimension 2 submanifold. But then (cf. \cite{mundet},) this defines a reduction $\pi$ of the structure group to a parabolic subgroup $P$, and an antidominant character $\chi$ of $P$ with $\deg_{a_i,\tau_i}(\pi,\chi)=\lambda ((A,\phi),-iu_\infty)\leq 0$, contradicting stability.

\subsection{Main estimate implies solution}
Here we need Lemma \ref{properintegral}, and the proof is essentially due to Bradlow \cite{bradlow}. Since $\Psi^c$ is proper in the weak topology, if $\Psi^c (\exp(s_j)$ is bounded, then $||s_j||_{L^2_p}$ is also bounded. But then, we take a minimizing sequence $(s_j)$ of $\Psi^c$, and by properness it converges weakly to some $s_\infty$ where $\Psi^c$ attains a minimum. But we've seen that minima to $\Psi^c$ correspond precisely to zeros of the moment map, so we only need to check smoothness, which follows from elliptic regularity.

\subsection{Solution implies stability}

Nothing really new happens here, since it mostly uses general properties of the integral of the moment map. It involves, however, computing some technical inequalities on the norms of the Lie algebra, so for details we refer to \cite{mundet}.

First of all, supposing that the orbit of a simple pair $(A,\phi)$ has a zero of the moment map, say $h\cdot (A,\phi)$, one proves that $h\cdot (A,\phi)$ is also a simple pair, and one with positive maximal weight. Indeed, a semisimple element contradicting stability of $h\cdot (A,\phi)$ could not leave it fixed, since this would contradict the simplicity of $(A,\phi)$ itself. By the explicit computation of the gradient of the moment map, we arrive then at a contradiciton. Now, using suitable inequalities, one proves that
\begin{equation*}
t\sup |g_{\pi,\chi}|\leq C_1 \Psi_{(A,\phi)}(\exp (g_{\pi,\chi})+C_2
\end{equation*}
It is standard from here to prove that the original pair is linearly stable.

\subsection{Uniqueness of solution}
The statement on uniqueness follows on general grounds from the convexity of the integral of the moment map.

\end{document}